\documentclass[12pt]{amsart}
\usepackage{tikz-cd}
\usepackage{amssymb,amsmath}
\usepackage{array}

\usepackage{amsmath, amssymb, amsfonts, amstext, amsthm, amscd}
\usepackage[mathscr]{euscript}
\usepackage{enumerate}

\usepackage{float}
\usepackage{graphicx}

\usepackage[raiselinks=false,colorlinks=true,citecolor=blue,urlcolor=blue,linkcolor=blue,bookmarksopen=true,pdftex]{hyperref}

\usepackage{tikz}
\usepackage{tikz-cd}
\usetikzlibrary{snakes}
\usetikzlibrary{intersections, calc}

\usepackage[english]{babel}
\usepackage{chngcntr}

\begin{document}

	\thispagestyle{empty}

	\def\theequation{\arabic{section}.\arabic{equation}}

	\newcommand{\codim}{\mbox{{\rm codim}$\,$}}
	\newcommand{\stab}{\mbox{{\rm stab}$\,$}}
	\newcommand{\lr}{\mbox{$\longrightarrow$}}

	\newcommand{\be}{\begin{equation}}
		\newcommand{\ee}{\end{equation}}
	
	\newtheorem{guess}{Theorem}[section]
	\newcommand{\bth}{\begin{guess}$\!\!\!${\bf }~}
		\newcommand{\eeth}{\end{guess}}
	\renewcommand{\bar}{\overline}
	\newtheorem{propo}[guess]{Proposition}
	\newcommand{\bpropo}{\begin{propo}$\!\!\!${\bf }~}
		\newcommand{\epropo}{\end{propo}}
	
	\newtheorem{lema}[guess]{Lemma}
	\newcommand{\blem}{\begin{lema}$\!\!\!${\bf }~}
		\newcommand{\elem}{\end{lema}}
	
	\newtheorem{defe}[guess]{Definition}
	\newcommand{\bdefe}{\begin{defe}$\!\!\!${\bf }~}
		\newcommand{\edefe}{\end{defe}}
	
	\newtheorem{coro}[guess]{Corollary}
	\newcommand{\bcor}{\begin{coro}$\!\!\!${\bf }~}
		\newcommand{\ecor}{\end{coro}}
	
	\newtheorem{rema}[guess]{Remark}
	\newcommand{\brem}{\begin{rema}$\!\!\!${\bf }~\rm}
		\newcommand{\erem}{\end{rema}}
	
	\newtheorem{exam}[guess]{Example}
	\newcommand{\beg}{\begin{exam}$\!\!\!${\bf }~\rm}
		\newcommand{\eeg}{\end{exam}}
	
	\newtheorem{notn}[guess]{Notation}
	\newcommand{\bnot}{\begin{notn}$\!\!\!${\bf }~\rm}
		\newcommand{\enot}{\end{notn}}
	\newcommand{\cw} {{\mathcal W}}
	\newcommand{\ct} {{\mathcal T}}
	\newcommand{\ch}{{\mathcal H}}
	\newcommand{\cf}{{\mathcal F}}
	\newcommand{\cd}{{\mathcal D}}
	\newcommand{\cR}{{\mathcal R}}
	\newcommand{\cv}{{\mathcal V}}
	\newcommand{\cn}{{\mathcal N}}
	\newcommand{\lra}{\longrightarrow}
	\newcommand{\ra}{\rightarrow}
	\newcommand{\blr}{\Big \longrightarrow}
	\newcommand{\da}{\Big \downarrow}
	\newcommand{\ua}{\Big \uparrow}
	\newcommand{\hra}{\mbox{{$\hookrightarrow$}}}
	\newcommand{\rt}{\mbox{\Large{$\rightarrowtail$}}}
	\newcommand{\dua}{\begin{array}[t]{c}
			\Big\uparrow \\ [-4mm]
			\scriptscriptstyle \wedge \end{array}}
	\newcommand{\ctext}[1]{\makebox(0,0){#1}}
	\setlength{\unitlength}{0.1mm}
	\newcommand{\cm}{{\mathcal M}}
	\newcommand{\cl}{{\mathcal L}}
	\newcommand{\cp}{{\mathcal P}}
	\newcommand{\ci}{{\mathcal I}}
	\newcommand{\bz}{\mathbb{Z}}
	\newcommand{\cs}{{\mathcal s}}
	\newcommand{\ce}{{\mathcal E}}
	\newcommand{\ck}{{\mathcal K}}
	\newcommand{\cz}{{\mathcal Z}}
	\newcommand{\cg}{{\mathcal G}}
	\newcommand{\cj}{{\mathcal J}}
	\newcommand{\cc}{{\mathcal C}}
	\newcommand{\ca}{{\mathcal A}}
	\newcommand{\cb}{{\mathcal B}}
	\newcommand{\cx}{{\mathcal X}}
	\newcommand{\co}{{\mathcal O}}
	\newcommand{\bq}{\mathbb{Q}}
	\newcommand{\bt}{\mathbb{T}}
	\newcommand{\bh}{\mathbb{H}}
	\newcommand{\br}{\mathbb{R}}
	\newcommand{\bl}{\mathbf{L}}
	\newcommand{\wt}{\widetilde}
	\newcommand{\im}{{\rm Im}\,}
	\newcommand{\bc}{\mathbb{C}}
	\newcommand{\bp}{\mathbb{P}}
	\newcommand{\ba}{\mathbb{A}}
	\newcommand{\spin}{{\rm Spin}\,}
	\newcommand{\ds}{\displaystyle}
	\newcommand{\tor}{{\rm Tor}\,}
	\newcommand{\bff}{{\bf F}}
	\newcommand{\bs}{\mathbb{S}}
	\def\ns{\mathop{\lr}}
	\def\nssup{\mathop{\lr\,sup}}
	\def\nsinf{\mathop{\lr\,inf}}
	\renewcommand{\phi}{\varphi}
	\newcommand{\tT}{{\widetilde{T}}}
	\newcommand{\tG}{{\widetilde{G}}}
	\newcommand{\tB}{{\widetilde{B}}}
	\newcommand{\tC}{{\widetilde{C}}}
	\newcommand{\tW}{{\widetilde{W}}}
	\newcommand{\tphi}{{\widetilde{\Phi}}}
	\newcommand{\fp}{{\mathfrak{p}}}

	\title{$K$-theory of Flag Bott manifolds}

	\author[B. Paul]{Bidhan Paul}
	\address{Department of Mathematics, Indian Institute of Technology, Madras, Chennai 600036, India}
	\email{bidhan@smail.iitm.ac.in}

	\author[V. Uma]{Vikraman Uma}
	\address{Department of Mathematics, Indian Institute of Technology, Madras, Chennai 600036, India
	}
	\email{vuma@iitm.ac.in}

	

	\begin{abstract}
		The aim of this paper is to describe the topological $K$-ring, in
		terms of generators and relations of a flag Bott manifold. We apply
		our results to give a presentation for the topological K-ring and
		hence the Grothendieck ring of algebraic vector bundles over flag
		Bott-Samelson varieties. 
	\end{abstract}

	\maketitle
	
	\section{Introduction}\label{flag}
	A Bott tower $B_{\bullet}=\{B_j\mid 1\leq j\leq m\}$, where
	$B_{j}=\mathbb{P}(1_{\mathbb{C}}\oplus L_j)$ for a line bundle $L_j$
	over $B_{j-1}$ for every $1\leq j$, is a sequence of
	$\mathbb{P}^1_{\mathbb{C}}$-bundles. Here $B_0=\ast$ is a point. At
	each stage $B_j$ is a $j$-dimensional nonsingular toric variety with
	an action of a dense torus $(\mathbb{C}^*)^{j}$. We call $B_j$ a
	$j$-stage Bott manifold for $1\leq j\leq m$.
	
	The topology and geometry of these varieties have been widely studied
	recently (see \cite{cms}). One main motivation to study these
	varieties comes from its relation to the Bott-Samelson variety which
	arise as desingularizations of Schubert varieties in the full flag
	variety(see \cite{bs}, \cite{d}, \cite{h}). Indeed it has been shown
	by Grossberg and Karshon (see \cite{gk}) that a Bott-Samelson variety
	admits a degeneration of complex structures with special fibre a Bott
	manifold (also see \cite{p}). Thus the topology and geometry of a Bott
	manifold gets related to those of the Bott -Samelson variety which in
	turn is equipped with rich connections with representation theory of
	semisimple Lie groups.
	
	The construction of a Bott tower can be generalized in several directions. One such natural direction is to replace $\mathbb{P}^1_{\mathbb{C}}$ with complex projective spaces of arbitrary dimensions. In other words at each stage $B_j$ is the projectivization of direct sum of finitely many complex line bundles. In this way we can construct the so called generalized Bott manifold and the generalized Bott tower which also has the structure of a nonsingular toric variety and has been studied widely (see \cite{cms1}, \cite{kl}).
	
	The construction of a Bott tower has recently been generalized in
	another natural direction of a {\it flag Bott manifold} by Kaji,
	Kuroki, Lee, Song and Suh (see \cite{kkls}, \cite{klss}, \cite{kur})
	by replacing $\mathbb{P}^1_{\mathbb{C}}$ (which can be identified with
	the variety of full flags in $\mathbb{C}^2$) at every stage with the
	variety $Flag(\mathbb{C}^n)$ of full flags in $\mathbb{C}^n$ for any
	$n\geq 2$. More precisely, at the $j$th stage we let $B_j$ to be the
	flagification of direct sum of $n_j+1$ complex line bundles over
	$B_{j-1}$. Thus $B_j$ is a flag bundle over $B_{j-1}$ with fibre the
	full flag manifold $Flag(\mathbb{C}^{n_j+1})$ of dimension
	$(n_j)(n_j+1)/2$.
	
	Moreover, in \cite{fls} the {\it flag Bott-Samelson variety} has been
	introduced and studied by Fujita, Lee and Suh. The flag Bott-Samelson
	varieties are special cases of the more general notion of the
	generalized Bott-Samelson variety which was introduced by Jantzen
	\cite{j} and studied by Perrin \cite{per} to obtain small resolutions
	of Schubert varieties. Earlier in \cite{sv}, Bott-Samelson varieties
	have been used by Sankaran and Vanchinathan to obtain resolutions of
	Schubert varieties in a partial flag variety $G/P$ by generalizing
	Zelevinsky's results who constructed small resolutions for all
	Schubert varieties in Grassmann varieties \cite{z}. The special property of the
	flag Bott-Samelson variety studied in \cite{fls} is that they are
	iterated bundles with fibres full flag varieties. In particular,
	similar to the Bott-Samelson variety the flag Bott-Samelson variety
	admits a one parameter family of complex structures which degenerates
	to a flag Bott manifold (see \cite[Section 4]{fls}). Thus the geometry
	and topology of flag Bott-Samelson variety gets related with that of
	flag Bott manifolds since the underlying differentiable structure is
	preserved under the deformation and can therefore be studied using
	this degeneration.

	There is a natural effective action of a complex torus $\mathbb{D}$
	and the corresponding compact torus $\mathbb{T}\subseteq \mathbb{D}$
	on the flag Bott manifold (see \cite{kkls}, \cite[Section 3.1]{klss}).
	
	The $\mathbb{T}$-equivariant cohomology of flag Bott manifolds has
	been studied by Kaji, Kuroki, Lee and Suh in \cite{kkls}.

	In \cite{cr} Civan and Ray gave presentations for any complex oriented
	cohomology ring (in particular the $K$-ring) of a Bott tower. They also
	determine the $KO$-ring for several families of Bott towers.
	
	In \cite{su2} P. Sankaran and the second named author describe the
	structure of the topological $K$-ring of a Bott tower, the topological
	$K$-ring of a Bott Samelson variety as well as the Grothendieck ring
	of a Bott Samelson variety.
	
	Our main aim in this paper is to study the topological $K$-theory of
	flag Bott manifolds. By applying the classical results in \cite{at}
	and \cite{k} for $K$-ring of a flag bundle iteratively, in Theorem
	\ref{mainth} we give a presentation for the $K$-ring of a flag Bott
	manifold in terms of generators and relations.
	
	In Corollary \ref{kfbs}, we apply our results to describe the
	topological $K$-ring of a flag Bott Samelson variety. This is
	essentially using the degeneration of the complex structures of a flag
	Bott Samelson variety to a flag Bott manifold.

	Also since the generating line bundles of the topological $K$-ring
	are in fact algebraic line bundles on the flag Bott-Samelson variety
	we show using \cite[Lemma 4.2]{su} that the $K$-ring of algebraic
	vector bundles is isomorphic to the topological $K$-ring. Thus
	Corollary \ref{kfbs} also gives a presentation for the Grothendieck
	ring of flag Bott-Samelson variety.

	Moreover, in Corollary \ref{kbs} we show that our presentation in
	particular generalizes the presentation of the $K$-ring of Bott
	manifolds and the $K$-ring and Grothendieck ring of Bott-Samelson
	varieties in \cite[Theorem 5.3, Theorem 5.4]{su2}.

	The flag Bott manifold has been defined for other Lie types by Kaji,
	Kuroki, Lee and Suh in \cite{kkls} and their equivariant cohomology
	ring has been computed. Thus in another direction our results are an
	extension of their results to the $K$-ring of flag Bott manifolds of
	the $A$ type.

	\subsection{Organisation of the sections}
	
	In section \ref{fbm} we recall the definition and construction of flag
	Bott manifolds from \cite{klss} and \cite{kur}. In particular, in
	subsection \ref{taut}, we recall the construction of tautological line
	bundles on these manifolds which generate its Picard group.
	
	In section \ref{fbs} we recall the definition of a flag Bott-Samelson
	variety and its associated flag Bott manifold from \cite{fls}.

	In section \ref{kth} we study the topological $K$-theory of flag Bott
	manifolds and flag Bott Samelson varieties.  In Theorem \ref{mainth}
	which is our first main result we give a presentation for the
	topological $K$-ring of a flag Bott manifold in terms of generators
	and relations.

	In subsection \ref{topk} we give the presentation for the topological
	$K$-ring and the Grothendieck ring of a flag Bott-Samelson variety as
	a corollary of Theorem \ref{mainth} (see Corollary \ref{kfbs}).

	\section{Flag Bott manifolds}\label{fbm}
	
	Let $E$ be an $n$ dimensional holomorphic vector bundle over a compact
	complex manifold $X$. We define the \textit{flag bundle} $Flag(E)$ as
	the bundle on $X$ whose fiber over each $x\in X$ is the full flag
	manifold ${Flag}(E_x)$. In particular,
	$\displaystyle Flag(E)=\bigsqcup_{x\in X}Flag(E_x)$. We recall the
	definition of flag Bott manifolds from \cite{kur, klss}. We broadly
	follow their notations and conventions.
	\begin{defe}\label{flBott}
		An $m$- stage flag Bott tower $B_\bullet=\{B_j|0\leq j\leq m\}$ is a sequence of manifolds
		$$ B_m\xrightarrow{p_m} B_{m-1}\xrightarrow{p_{m-1}}.\,.\,.\,. \xrightarrow{p_2} B_1 \xrightarrow{p_1}B_0 = \{\ast\}$$ which is defined recursively as follows :\\
		\begin{enumerate}
			\item $B_0$ is a point.
			\item $\displaystyle B_j:={Flag}(\bigoplus_{l=1}^{n_j+1}\eta^{(j)}_l)$, where $\eta^{(j)}_l$ is a holomorphic line bundle  over $B_{j-1}$ for each $1\leq l\leq n_j+1$ and $1\leq j\leq m$.
		\end{enumerate} We call $B_j$ as $j-$stage flag Bott manifold of
		the flag Bott tower $B_\bullet$.
	\end{defe}
	Two flag Bott towers $B_\bullet=\{B_j\,|\,0\leq j\leq m\}$ and $B'_\bullet=\{B'_j\,|\,0\leq j\leq m_1\}$ are said to be isomorphic if $m=m_1$ and 
	$$\begin{tikzcd}
		B_j \arrow{r}{\phi_j} \arrow[swap]{d}{p_j} & B'_j \arrow{d}{p'_j} \\
		B_{j-1} \arrow{r}{\phi_{j-1}}& B'_{j-1}
	\end{tikzcd}$$ there exist a collection of diffeomorphisms $\phi_j:B_j\to B'_j$ for each $1\leq j\leq m$ such that 
	the above diagram commutes for each $1\leq j\leq m$.
	
	First we give some trivial examples of flag Bott manifolds. Later we
	shall give a non-trivial example.
	\begin{exam}
		\begin{enumerate}
			\item The flag manifold ${Flag}(n+1):={Flag}(\bc^{n+1})$ is a 1-stage flag Bott tower.
			\item The usual product of flag manifolds ${Flag}(n_1+1)\times {Flag}(n_2+1)\times \cdots\times {Flag}(n_l+1)$ is an $l$- stage flag Bott manifold.
		\end{enumerate}
	\end{exam}
	
	We now define the $j$-stage flag Bott manifold for each $1\leq j\leq m$  as an orbit space under certain  right action 
	$$B_j^{quo}:= \prod_{l=1}^{j} GL(n_l+1)/\prod_{l=1}^{j} B_{GL(n_l+1)}$$ where $B_{GL(n_l+1)}$ denotes the Borel subgroup consisting of upper triangular matrices of $GL(n_l+1)$ for $l=1,2,\ldots,j$.\\
	$B_1^{quo}$ is the flag manifold ${Flag}(n_1+1)= GL(n_1+1)/B_{GL(n_1+1)}$. In order to define the flag Bott manifolds of higher stages we need a sequence of matrices with integer entries
	\begin{equation}
		\mathfrak{P}:=(P_l^{(j)})_{1\leq l<j\leq m}\in \prod_{1\leq l<j\leq m} M_{n_j+1,n_l+1}(\bz).
	\end{equation}
	Let $D(n_i+1)\subset GL(n_i+1)$ denotes the collection of diagonal matrices in $GL(n_i+1)$ for each $i=1,2,\ldots,m$. Each $P_l^{(j)}$ for $1\leq l<j\leq m$ encodes a $B_{GL(n_l+1)}$ action on $GL(n_j+1)$ as follows. 
	
	Let $$P_l^{(j)}=\begin{bmatrix}
		\bf{a}_1\\ \bf{a}_2\\ .\\ .\\.\\ \bf{a}_{n_j+1}
	\end{bmatrix}=\begin{bmatrix}
		a_{11}& a_{12}&...& a_{1,n_l+1}\\ a_{21}& a_{22}&...& a_{2,n_l+1}\\.&.&...&.\\ .&.&...&.\\.&.&...&.\\ a_{n_j+1,1}& a_{n_j+1,2}&...& a_{n_j+1,n_l+1}
	\end{bmatrix}\in M_{n_j+1,n_l+1}(\bz) $$
	
	Since the character group $X^*(D(n_l+1))$ is isomorphic to $\bz^{n_l+1}$, we define a group homomorphism $\pi_l^{(j)}: D(n_l+1)\to D(n_j+1)$ using the matrix $P_l^{(j)}$ i.e 
	\begin{equation}
		\pi_l^{(j)}:  h \to diag(h^{\bf{a_1}}, h^{\bf{a_2}},\ldots,h^{\bf{a_{n_j+1}}})\in D(n_j+1) 
	\end{equation}
	where, $ h^{\bf{a_i}}=h_1^{a_{i1}}\cdot h_2^{a_{i2}}\cdots h_{n_l+1}^{a_{in_l+1}}$ for $h :=diag(h_1,h_2,\ldots,h_{n_l+1})\in D(n_l+1)$ and ${\bf{a_i}}=(a_{i1},a_{i2},\ldots,a_{in_l+1})$ is the $i$-th row vector of $P_l^{(j)}$ for each $1\leq i\leq n_j+1$.
	
	We now define the homomorphism $\Psi_l^{(j)}:=\Psi(P_l^{(j)}) :B_{GL(n_l+1)}\to D(n_j+1)$ by
	\begin{equation}
		\Psi_l^{(j)}= \pi_l^{(j)}\circ\Gamma_l
	\end{equation}
	i.e for each $b\in B_{GL(n_l+1)}$, 
	$$\Psi_l^{(j)}(b)= diag(\Gamma_l(b)^{\bf{a_1}},\Gamma_l(b)^{\bf{a_2}},\ldots,\Gamma_l(b)^{\bf{a_{n_j+1}}})\in D(n_j+1) $$
	where $\Gamma_l : B_{GL(n_l+1)}\to D(n_l+1)$ is the canonical projection for each $1\leq l\leq m$.
	
	Now we define a right $\displaystyle\prod_{l=1}^{j} B_{GL(n_l+1)}$ -action 
	\begin{equation}\label{ract}
		\Phi_j^{\mathfrak{P}} : \prod_{l=1}^{j} GL(n_l+1)\times\prod_{l=1}^{j} B_{GL(n_l+1)} \longrightarrow \prod_{l=1}^{j} GL(n_l+1)\end{equation} by $\Phi_j^{\mathfrak{P}}((g_1,g_2,\ldots,g_j),(b_1,b_2,\ldots,b_j))
	:=$
	\begin{align*}
		\big(g_1b_1, {(\Psi_1^{(2)}(b_1))}^{-1}g_2b_2,
		{(\Psi_1^{(3)}(b_1))}^{-1}{(\Psi_2^{(3)}(b_2))}^{-1}g_3b_3,\\
		\ldots,{(\Psi_1^{(j)}(b_1))}^{-1}{(\Psi_2^{(j)}(b_2))}^{-1}\cdots {(\Psi_{j-1}^{(j)}(b_{j-1}))}^{-1}g_jb_j\big)
	\end{align*}
	\begin{lema}(\cite[Lemma 2.6]{klss})\label{frpr}
		$\Phi_j^{\mathfrak{P}}$ in \eqref{ract} is a free and proper right
		action for $1\leq j\leq m$.
	\end{lema}
	Hence by Lemma \ref{frpr}, the orbit space \[B_j^{quo}(\mathfrak{P}):=\prod_{l=1}^{j} GL(n_l+1)/\Phi_j^{\mathfrak{P}} \] is a complex manifold as $\displaystyle\prod_{l=1}^{j} B_{GL(n_l+1)}$ acts freely and properly on the complex manifold $\displaystyle\prod_{l=1}^{j} GL(n_l+1)$ (see \cite[Proposition 2.1.13]{huy}) and $B^{quo}_\bullet(\mathfrak{P}) :=\{B_j^{quo}(\mathfrak{P})\,|\, 0\leq j\leq m\} $ is  a flag Bott tower of height $m$ (see \cite[Proposition 2.7]{klss}).
	
	We write an element of $B_j^{quo}(\mathfrak{P})$ as $[g_1,g_2,\ldots, g_j]$ which is the orbit of $\displaystyle (g_1,g_2,\ldots, g_j)\in \prod_{l=1}^{j} GL(n_l+1)$ under $\displaystyle\prod_{l=1}^{j} B_{GL(n_l+1)}$ action and $p_j: B_j^{quo}\to B_{j-1}^{quo}$ is defined by $$[g_1,g_2,\ldots, g_{j-1},\, g_j]\mapsto [g_1,g_2,\ldots, g_{j-1}].$$
	
	Since the character group
	$\displaystyle X^*(\prod_{l=1}^{j}{D(n_l+1)})\simeq
	\bigoplus_{l=1}^{j}\bz^{n_l+1}$, we define a holomorphic line bundle
	on $B_j^{quo}$ for each integer vector
	$\displaystyle ({\bf{v}}_1,{\bf{v}}_2,\ldots,{\bf{v}}_j)\in
	\bigoplus_{l=1}^{j}\bz^{n_l+1}$ as an orbit space :
	\begin{equation}\label{line}
		\eta({\bf{v}}_1,{\bf{v}}_2,\ldots,{\bf{v}}_j):=\bigg(\prod_{l=1}^{j} GL(n_l+1)\times \mathbb{C}\bigg)/\prod_{l=1}^{j} B_{GL(n_l+1)}
	\end{equation}

	where we have the following right action
	\[(g_1,g_2,\ldots,
	g_j,w).(b_1,b_2,\ldots,b_j)\\:=(\Phi_j^{\mathfrak{P}}((g_1,g_2,\ldots,
	g_j).(b_1,b_2,\ldots, b_j)),b_1^{-{\bf{v}}_1}\cdots
	b_j^{-{\bf{v}}_j}w).\] Here  \[b_l^{-{\bf v}_l}:=\Gamma_l(b_l)^{-{\bf
			v}_l}\] for $1\leq l\leq j$ and \[h^{\bf v}:= h_1^{v_1}\cdots
	h_{n_l+1}^{v_{n_l+1}}\] for $h\in D(n_l+1)$ and ${\bf v}=(v_{1},\ldots,
	v_{n_l+1})\in \mathbb{Z}^{n_l+1}$.

	\begin{propo}\label{prop}(see \cite[proof of Prop.2.7]{klss})
		For a collection of matrices $$\mathfrak{P}:=(P_l^{(j)})_{1\leq l<j\leq m}\in \prod_{1\leq l<j\leq m}M_{n_j+1,n_l+1}(\bz)$$ the $j$-stage flag Bott manifold $B_j^{quo}$ of the flag Bott tower $B_\bullet^{quo}(\mathfrak{P})$ is the induced flag bundle $Flag(\eta^{(j)})$ over $B_{j-1}^{quo}$, where $$\eta^{(j)}:=\bigoplus_{k=1}^{n_j+1}\eta({\bf{v}}^{(j)}_{k,1},{\bf{v}}^{(j)}_{k,2},\ldots,{\bf{v}}^{(j)}_{k,j-1})$$ and ${\bf{v}}^{(j)}_{k,l}$ is the $k$-th row vector of the matrix $P^{(j)}_l$ for each $1\leq l\leq j-1$.
	\end{propo}
	
	\brem\label{dslbs}
	Note that $\eta^{(j)}$ is a direct sum of $n_j+1$ complex line bundles
	$\eta^{(j)}_l$ for $1\leq l\leq n_j+1$ on $B_{j-1}^{quo}$. Thus the
	structure group of $\eta^{(j)}$ is the complex torus
	$D(n_j+1)$. In particular this implies that for $\bf v$ in the fiber
	$(\eta^{(j)})_{[g_1,\ldots, g_{j-1}]}$ of $\eta^{(j)}$ over the point
	$[g_1,\ldots, g_{j-1}]\in B_{j-1}^{quo}$ we have the identification
	\[\Psi_1^{(j)}(b_1)^{-1}\cdots
	\Psi_{j-1}^{(j)}(b_{j-1})^{-1}\cdot {\bf v}\sim {\bf v}\] since
	$\Psi_l^{(j)}(b_l)\in D(n_j+1)$ for every $1\leq l\leq j-1$.
	\erem

	{\it Proof of Proposition \ref{prop}.}  Consider the map
	$\phi_j :B_j^{quo}\to Flag(\eta^{(j)})$ defined
	by
	$$[g_1,\,g_2,\ldots,g_{j-1},\,g_j]\mapsto ([g_1,\,g_2,\ldots,g_{j-1}],
	\,\underline{V})$$ where
	$\underline{V}=(V_1\subsetneq V_2\subsetneq \cdots\subsetneq V_{n_j}
	\subsetneq ( \eta^{(j)}_{[g_1, g_2,\ldots,g_{j-1}]})$ is the full flag
	of $\eta^{(j)}_{[g_1, g_2,\ldots,g_{j-1}]}$ such that $V_k$ is spanned
	by first $k$ columns of $g_j\in GL(n_j+1)$. Note that
	\begin{equation*}
		\begin{split}
			[\Phi_j^{\mathfrak{P}}((g_1,\ldots,g_j),(b_1,\ldots,b_j))]
			\mapsto &([\Phi_{j-1}^{\mathfrak{P}}((g_1,\ldots,g_{j-1}),(b_1,\ldots,b_{j-1}))]\,,\underline{V'}) \\
			&= ([g_1,g_2,\ldots,g_{j-1}], \underline{V'})
		\end{split}
	\end{equation*}
	for
	$\displaystyle (b_1,b_2,\ldots,b_j)\in \prod_{l=1}^jB_{GL(n_j+1)}$.
	Here
	\[\underline{V'}=(V'_1\subsetneq V'_2\subsetneq \cdots \subsetneq
	V'_{n_j} \subsetneq ( \eta^{(j)}_{[g_1, g_2,\ldots,g_{j-1}]})\] is
	the full flag of $\eta^{(j)}_{[g_1, g_2,\ldots,g_{j-1}]}$ such that
	$V'_k$ is spanned by the first $k$ columns of
	${(\Psi_1^{(j)}(b_1))}^{-1}{(\Psi_2^{(j)}(b_2))}^{-1}\cdots
	{(\Psi_{j-1}^{(j)}(b_{j-1}))}^{-1}g_j\cdot b_j\,\in GL(n_j+1)$. Since
	$b_j\in B_{GL(n_j+1)}$, the vector space spanned by the first $k$
	column vectors of $g_j\cdot b_j$ is $V_k$ for $1\leq k\leq
	n_j+1$. Now, the column vectors of $g_j$ span the fibre
	$(\eta^{(j)})_{[g_1,\ldots, g_{j-1}]}$. Thus by Remark \ref{dslbs} it
	follows that we can identify any column vector $\bf v$ of $g_j$ with
	$\Psi_1^{(j)}(b_1)^{-1}\cdots \Psi_{j-1}^{(j)}(b_{j-1})^{-1}\cdot {\bf
		v}$ in $(\eta^{(j)})_{[g_1,\ldots, g_{j-1}]}$. Thus the flags
	$\underline{V'}$ and $\underline{V}$ in
	$(\eta^{(j)})_{[g_1,\ldots, g_{j-1}]}$ can be identified. It follows
	that $([g_1,g_2,\ldots,g_{j-1}], \underline{V'})$ and
	$([g_1,g_2,\ldots,g_{j-1}], \underline{V})$ can be identified as
	elements of $Flag(\eta^{(j)})$. Hence $\phi_j$ is well defined.

	Let $f_j: Flag(\eta^{(j)})\to B_j^{quo}$ be the map defined by
	$$([g_1, g_2,\ldots,g_{j-1}],\underline{V})\mapsto [g_1,
	g_2,\ldots,g_{j-1}, g_j]$$  where
	\[\underline{V}=(V_1\subsetneq V_2\subsetneq ...\subsetneq V_{n_j}
	\subsetneq ( \eta^{(j)}_{[g_1, g_2,\ldots,g_{j-1}]})\] is the full
	flag of $\eta^{(j)}_{[g_1, g_2,\,.\,.\,.\,,g_{j-1}]}$ and
	$g_j\in GL(n_j+1)$ is the matrix such that the first $k$ columns span
	the vector space $V_k$ for $1\leq k\leq n_j+1$. Then $f_j$ is the
	inverse of $\phi_j$. It suffices therefore to show that $f_j$ is well
	defined. This follows since the element
	$([\Phi_{j-1}^{\mathfrak{P}}((g_1,\ldots,g_{j-1}),(b_1,\ldots,b_{j-1}))],
	\underline{V})$ maps to
	$[\Phi_{j}^{\mathfrak{P}}((g_1,\ldots,g_{j}),(b_1,\ldots,b_{j}))]=[g_1,\ldots,g_{j}]$.
	This again follows because the span of the first $k$ column vectors of
	$g_j$ can be identified with the span of the first $k$ column vectors
	of
	\[\Psi_1^{(j)}(b_1)^{-1}\cdots \Psi_{j-1}^{(j)}(b_{j-1})^{-1} \cdot
	g_j\cdot b_j.\] Hence the map $\phi_j$ is a
	diffeomorphism. Therefore the proposition follows. \hfill $\Box$

	\begin{exam}\label{ex}
		For $n_1=2,\, n_2=1,\,$ and $n_3=1$, let 
		$$P^{(2)}_1=\begin{bmatrix}a_1 & a_2 & a_3\\ b_1 & b_2 & b_3
		\end{bmatrix};
		P^{(3)}_1=\begin{bmatrix}c_1 & c_2 & c_3\\ d_1 & d_2 & d_3
		\end{bmatrix}; P^{(3)}_2=\begin{bmatrix}f_1 & f_2\\ 0 & 0
		\end{bmatrix} $$ be the collection of matrices which determines the
		right action $\Phi_j^{\mathfrak{P}}$ of
		$\prod_{l=1}^{j} B_{GL(n_l+1)}$ on $\prod_{l=1}^{j} GL(n_l+1)$ as in
		\eqref{ract} for $j=1,2,3$. By Proposition \ref{prop},
		$B_\bullet^{quo}(\mathfrak{P})$ is isomorphic to the $3$ stage flag
		Bott tower
		\[B_3\xrightarrow{p_3} B_2\xrightarrow{p_2} B_1\xrightarrow{p_1}
		B_0=\{a\,\,point\}\] where the corresponding flag Bott manifolds are:\\
		$B_1=Flag(3)$\\
		$B_2=Flag(\eta((a_1,\,a_2,\,a_3))\oplus\eta((b_1,\,b_2,\,b_3)))$\\
		$B_3=Flag(\eta((c_1,\,c_2,\,c_3),(f_1,\,f_2))\oplus\eta((d_1,\,d_2,\,d_3),(0,\,0)))$.
		
		\noindent
		The line bundle $\eta((c_1,\,c_2,\,c_3),(f_1,\,f_2))$ over $B_2$ is
		\[(GL(3)\times GL(2)\times \bc)/(B_{GL(3)}\times B_{GL(2)})\] where the
		right action of $(B_{GL(3)}\times B_{GL(2)})$ is given by
		$$(g_1,g_2,w)\cdot(b_1,b_2):=\bigg(\Phi^{\mathfrak{P}}_2((g_1,g_2)\cdot(b_1,b_2)),\,\,
		b_1^{-(c_1,\,c_2,\,c_3)}b_2^{-(f_1,f_2)}w\bigg)$$ (see \eqref{line}). 
	\end{exam}

	\begin{guess}(\cite[Theorem 2.10]{klss})
		For any flag Bott tower $B_\bullet$ of height $m$, there is a sequence
		of matrices $$\mathfrak{P}:=(P_l^{(j)})_{1\leq l<j\leq m}\in
		\prod_{1\leq l<j\leq m} M_{n_j+1,n_l+1}(\bz)$$ such that
		$B_\bullet^{quo}(\mathfrak{P})$ is isomorphic to $B_\bullet$ as flag
		Bott towers. 
	\end{guess}
	
	\begin{defe}\label{defquo}
		A flag Bott tower $B_\bullet$ is said to be determined by the
		collection of matrices
		$\mathfrak{P}:=(P_l^{(j)})_{1\leq l<j\leq m}\in \prod_{1\leq l<j\leq
			m}M_{n_j+1,n_l+1}(\mathbb{Z})$ if $B_\bullet$ is isomorphic to
		$B_\bullet^{quo}(\mathfrak{P})$ as flag Bott towers.
	\end{defe}

	\subsection{Tautological line bundles over a flag Bott manifold} \label{taut}
	
	We recall from Definition \ref{flBott}, that any point of
	$B_j=Flag(\eta^{(j)})$, where we write $\displaystyle \eta^{(j)}:=\bigoplus_{l=1}^{n_j+1}\eta^{(j)}_l$, can be interpreted as $$(p,\underline{V})=\bigg(p,\,V_0\subset V_1\subset V_2\subset \cdots\subset V_{n_j}\subset V_{n_j+1}=\eta^{(j)}_p\bigg)$$ for $p\in B_{j-1}$. For each $1\leq k\leq n_j+1$, we define a sub bundle $W_{j,k}\subseteq p_j^*(\eta^{(j)})$  over $B_j$ which has fiber $V_k$ of the flag $\underline{V}$ over a point $(p,\underline{V})\in B_j$. Hence we have the quotient line bundle $W_{j,k}/W_{j,k-1}$ over $B_j$ for each $1\leq k\leq n_j+1$. \\

	\begin{lema}\label{line2}(see \cite[Lemma 2.12]{klss})
		Let $B^{quo}_\bullet:= B^{quo}_\bullet(\mathfrak{P})$ be a flag Bott
		tower for a sequence of matrices
		$\displaystyle\mathfrak{P}:=(P_l^{(j)})_{1\leq l<j\leq m}\in
		\prod_{1\leq l<j\leq m} M_{n_j+1,n_l+1}(\bz)$ defined as in
		\eqref{ract}. Then the line bundle $W_{j,k}/W_{j,k-1}\to B_j^{quo}$
		is isomorphic to the line bundle
		$\eta({\bf{0}},{\bf{0}},\ldots,{\bf{0}}, {\bf{e}}_k)\to B_j^{quo}$
		as defined in \eqref{line}, where
		${\bf{e}}_k:= (0,\ldots,0,1,0,\ldots,0)\in \bz^{n_j+1}$ has $1$ in
		the $k$-th place with all other entries zero.
	\end{lema}

	\begin{proof}
		By Proposition \ref{prop}, any element
		$g=[g_1,\,g_2,\ldots,g_j]\in B_j^{quo}$ can be considered as a full
		flag
		$$(p_j(g),\,\underline{V}) :=\Big(p_j(g),\, V_0\subset V_1\subset
		V_2\subset \cdots\subset V_{n_j}\subset
		V_{n_j+1}=\eta^{(j)}_{p_j(g)}\Big)$$ where
		$\eta^{(j)}:=\oplus_{k=1}^{n_j+1}\eta({\bf{v}}^{(j)}_{k,1},{\bf{v}}^{(j)}_{k,2},\ldots,{\bf{v}}^{(j)}_{k,j-1})$
		and ${\bf{v}}^{(j)}_{k,l}$ is the $k$-th row vector of the matrix
		$P^{(j)}_l$ for each $1\leq l\leq j-1$. The fiber of $W_{j,k}$ at
		$g\in B_j^{quo}$ is the vector space
		$V_k\subset \eta^{(j)}_{p_j(g)}$ which is spanned by first $k$
		columns of $g_j\in GL(n_j+1)$ say,
		${\bf{u}}_1,\, {\bf{u}}_2,\ldots,{\bf{u}}_k \in \eta
		^{(j)}_{p_j(g)}$ . Hence the fiber of the canonical line bundle
		$W_{j,k}/W_{j,k-1}$ at $g$ is $V_k/V_{k-1}$ which is spanned
		by the class $\bar{{\bf{u}}_k}\in V_k/V_{k-1}$ of ${\bf{u}}_k\in \eta ^{(j)}_{p_j(g)}$.\\
		Let $b=(b_1,\ldots,b_j)$ be an element of
		$\displaystyle\prod_{i=1}^jB_{GL(n_i+1)}$. Then it can be seen that the class
		$\bar{{\bf{u}}'_k}\in V_k/V_{k-1}$ of the $k$-th column vector
		${\bf{u}}'_k$ of the last coordinate
		$$ {(\Psi_1^{(j)}(b_1))}^{-1}{(\Psi_2^{(j)}(b_2))}^{-1}\cdots
		{(\Psi_{j-1}^{(j)}(b_{j-1}))}^{-1}g_jb_j$$ of
		$\Phi_j^{\mathfrak{P}}(g,\,b)$ is equal to
		$b_j^{{\bf{e}}_k}\cdot \bar{{\bf{u}}_k}$. This is because the $k$th
		column vector ${\bf{u}}'_k$ is in the span of
		${\bf{u}}_1,\, {\bf{u}}_2,\ldots, {\bf{u}}_k $ where the coefficient
		of ${\bf{u}}_k$ is
		\[ {(\Psi_1^{(j)}(b_1))}^{-1}{(\Psi_2^{(j)}(b_2))}^{-1}\cdots
		{(\Psi_{j-1}^{(j)}(b_{j-1}))}^{-1}\cdot b_j^{{\bf{e}}_k}.\] Now, by
		Remark \ref{dslbs} we have the
		equivalence
		$${(\Psi_1^{(j)}(b_1))}^{-1}{(\Psi_2^{(j)}(b_2))}^{-1}\cdots
		{(\Psi_{j-1}^{(j)}(b_{j-1}))}^{-1}\cdot b^{{\bf{e}}_k}\cdot
		{\bf{u}}_k\,\sim\,b^{{\bf{e}}_k}\cdot {\bf{u}}_k$$ in
		$(\eta^{(j)})_{[g_1,\ldots, g_{j-1}]}$.  The result now follows by
		the definition (see \eqref{line}) of the holomorphic line bundle
		$\eta({\bf{0}},{\bf{0}},\ldots,{\bf{0}}, {\bf{e}}_k)$ on
		$B_j^{quo}$.
	\end{proof}

	\begin{coro}\label{cor}
		The line bundle $p_j^*\circ \cdots\circ p^*_{l+1}(W_{l,k}/W_{l,k-1})\to B_j^{quo}$ is isomorphic to the line bundle $$\eta({\bf{0}},{\bf{0}},\ldots,{\bf{0}}, {\bf{e}}_k,\underbrace{{\bf{0}},{\bf{0}},\ldots,{\bf{0}}}_{j-l})\to B_j^{quo}$$ as defined in \eqref{line}, where ${\bf{e}}_k:= (0,0,\ldots,1,0,\ldots,0)\in \bz^{n_l+1}$ has $1$ in the $k$-th place with all other entries zero.
	\end{coro}
	\begin{proof}
		Let $\eta({\bf{v}}_1,{\bf{v}}_2,\ldots,{\bf{v}}_l)$ be any
		holomorphic line bundle over $B_l^{quo}$ for an integer vector
		$\displaystyle({\bf{v}}_1,{\bf{v}}_2,\ldots,{\bf{v}}_l)\in
		\bigoplus_{i=1}^{l}\bz^{n_i+1}$ as defined in \eqref{line}. Then
		$p_{l+1}^*\eta({\bf{v}}_1,{\bf{v}}_2,\ldots,{\bf{v}}_l)$ is a
		holomorphic line bundle over $B_{l+1}^{quo}$, whose fiber over a
		point $[g_1,\ldots,g_l,\,g_{l+1}]\in B_{l+1}^{quo}$ is
		$\eta({\bf{v}}_1,{\bf{v}}_2,\ldots,{\bf{v}}_l)_{[g_1,\ldots,\,g_l]}$
		and, is defined by
		$(g_1,g_2,\ldots, g_{l+1},w)\cdot(b_1,b_2,\ldots,b_l,\,b_{l+1})$
		
		$:=\big(\Phi_{l+1}^{\mathfrak{P}}((g_1,g_2,\ldots, g_{l+1})\cdot(b_1,b_2,\ldots, b_{l+1})),b_1^{-{\bf{v}}_1}\cdots b_l^{-{\bf{v}}_l}w\big)$
		
		$ = \big(\Phi_{l+1}^{\mathfrak{P}}((g_1,g_2,\ldots, g_{l+1})\cdot(b_1,b_2,\ldots, b_{l+1})),b_1^{-{\bf{v}}_1}\cdots b_l^{-{\bf{v}}_l}b_{l+1}^{-\bf{0}}w\big)$
		
		for $(g_1,g_2,\ldots, g_{l+1}, w)\in \prod_{i=1}^{l+1} GL(n_i+1)\times \bc$ and  $(b_1,b_2,\,.\,.\,, b_{l+1})\in  \prod_{i=1}^{l+1} B_{GL(n_i+1)}$.
		
		Therefore, $p_{l+1}^*\eta({\bf{v}}_1,{\bf{v}}_2,\ldots,{\bf{v}}_l)$ is isomorphic to $\eta({\bf{v}}_1,{\bf{v}}_2,\ldots,{\bf{v}}_l, {\bf{0}})$ as line bundles over $B_{l+1}^{quo}$. Hence, the result follows from Theorem \ref{line2} by repeating the above procedure for higher $j>l+1$. 
	\end{proof}

	\begin{lema}\label{tensor}
		Let $\eta({\bf{u}}_1,{\bf{u}}_2,\ldots,{\bf{u}}_j)$ and
		$\eta({\bf{v}}_1,{\bf{v}}_2,\ldots,{\bf{v}}_j)$ be two line bundles
		over $B_j^{quo}$ for
		$\displaystyle ({\bf{u}}_1,{\bf{u}}_2,\ldots,{\bf{u}}_j)$,
		$({\bf{v}}_1,{\bf{v}}_2,\ldots,{\bf{v}}_j)\in
		\bigoplus_{l=1}^{j}\bz^{n_l+1}$. Then
		$\eta({\bf{u}}_1+{\bf{v}}_1,{\bf{u}}_2+{\bf{v}}_2,\ldots,{\bf{u}}_j+{\bf{v}}_j)\simeq
		\eta({\bf{u}}_1,{\bf{u}}_2,\ldots,{\bf{u}}_j)\otimes\eta({\bf{v}}_1,{\bf{v}}_2,\ldots,{\bf{v}}_j)$
		are isomorphic as line bundles over $B_j^{quo}$.
	\end{lema}
	\begin{proof}
		Let $\eta({\bf{u}}_1,{\bf{u}}_2,\ldots,{\bf{u}}_j)$ be a line bundle  over $B_j^{quo}$ as defined in \eqref{line}, respectively $\eta({\bf{v}}_1,{\bf{v}}_2,\ldots,{\bf{v}}_j)$. In which, right $\displaystyle\prod_{l=1}^{j} B_{GL(n_l+1)}$- action over $\big(\displaystyle\prod_{l=1}^{j} GL(n_l+1)\times \mathbb{C}\big)$ is given as follows :\\
		$(g_1,\ldots, g_j,w_1)\cdot(b_1,\ldots, b_j):=(\Phi_j^{\mathfrak{P}}((g_1,\ldots, g_j).(b_1,\ldots, b_j)),b_1^{-{\bf{u}}_1}\cdots b_j^{-{\bf{u}}_j}w_1)$.
		\noindent
		respectively, \\
		$(g_1,\ldots, g_j,w_2)\cdot(b_1,\ldots, b_j):=(\Phi_j^{\mathfrak{P}}((g_1,\ldots, g_j).(b_1,\ldots, b_j)),b_1^{-{\bf{v}}_1}\cdots b_j^{-{\bf{v}}_j}w_2)$\\
		for $\displaystyle (g_1,g_2,\ldots, g_j)\in \prod_{l=1}^{j}GL(n_l+1)$,
		$\displaystyle(b_1,b_2,\ldots, b_j)\in \prod_{l=1}^{j} B_{GL(n_l+1)}$ and $w_1, w_2\in\bc$.\\
		
		\noindent
		Hence for $\eta({\bf{u}}_1,{\bf{u}}_2,\ldots,{\bf{u}}_j)\otimes\eta({\bf{v}}_1,{\bf{v}}_2,\ldots,{\bf{v}}_j)$, we have the following right $\displaystyle\prod_{l=1}^{j} B_{GL(n_l+1)}$- action over $\displaystyle\big(\prod_{l=1}^{j} GL(n_l+1)\times \mathbb{C}\big)$ :
		\begin{equation*}
			\begin{split}
				&  (g_1,g_2,\ldots, g_j,w_1\otimes w_2)\cdot(b_1,b_2,\ldots, b_j) =\\
				& \big(\Phi_j^{\mathfrak{P}}((g_1,g_2,\ldots, g_j)\cdot(b_1,b_2,\ldots, b_j))\,,\,b_1^{-{\bf{u}}_1}\cdots b_j^{-{\bf{u}}_j}w_1\,\otimes\,b_1^{-{\bf{v}}_1}\cdots b_j^{-{\bf{v}}_j}w_2\big)\\
				=& \big(\Phi_j^{\mathfrak{P}}((g_1,g_2,\ldots, g_j)\cdot(b_1,b_2,\ldots, b_j))\,,\,b_1^{-{\bf{u}}_1-{\bf{v}}_1}\cdots b_j^{-{\bf{u}}_j-{\bf{v}}_j}w_1\otimes w_2\big)\\
				=& \big(\Phi_j^{\mathfrak{P}}((g_1,g_2,\ldots, g_j)\cdot(b_1,b_2,\ldots, b_j))\,,\,b_1^{-({\bf{u}}_1+{\bf{v}}_1)}\cdots b_j^{-({\bf{u}}_j+{\bf{v}}_j)}w_1\otimes w_2\big).
			\end{split}
		\end{equation*}
		The second equality above follows since
		$b_1^{-{\bf{u}}_1}\cdots b_j^{-{\bf{u}}_j}$ and
		$b_1^{-{\bf{v}}_1}\cdots b_j^{-{\bf{v}}_j}$ are both scalars.  Hence
		the result follows from \eqref{line}.
	\end{proof}

	\begin{notn}
		Let $\mathcal{L}_{j,k}:=W_{j,k}/W_{j,k-1}$ for $1\leq k\leq n_j+1$ and
		$1\leq j\leq m$.
	\end{notn}

	\brem\label{algebraicline bundles} We note that in the construction of
	flag Bott manifolds $B_j$ and the line bundles
	$\eta({\bf{v}}_1,{\bf{v}}_2,\ldots,{\bf{v}}_j)$ (see \eqref{ract}) and
	\eqref{line}) that all the objects are algebraic varieties and all the
	morphisms are algebraic morphisms. Thus by construction the flag Bott
	manifolds $B_j$, $1\leq j\leq m$ are complex algebraic varieties and
	the tautological line bundles $\mathcal{L}_{j,k}$ for
	$1\leq k\leq n_j+1$ and $1\leq j\leq m$ are complex algebraic line
	bundles.  \erem

	\begin{propo}(see \cite[Lemma 2.11]{klss})
		Let $B_\bullet$ be a flag Bott tower. Then the $Pic(B_j)$ is generated by the set of line bundles $$\{\mathcal{L}_{j,k}\,|\, 1\leq k\leq n_j+1\}\,\cup\, \bigcup_{l=1}^{j-1}\{p_j^*\,\circ\, \cdots\,\circ\, p^*_{l+1}(\mathcal{L}_{l,k}) \,|\, 1\leq k\leq n_l+1\}$$ for each $1\leq j\leq m$.
	\end{propo}
	\begin{proof}
		Let $X$ be a flag bundle over $Y$. Recall from (Example 19.1.11, \cite{ful}), that the cycle map $cl_X: A_k(X)\to H_{2k}(X)$ is an isomorphism if and only if $cl_Y$ is an isomorphism. Moreover, the cycle map is an isomorphism for an arbitrary flag manifold. Since, $B_j$ is an iterated bundle of flags over a point, the cycle map $cl_{B_j}: A_k({B_j})\to H_{2k}({B_j})$ is an isomorphism. On the other hand, since flag Bott manifolds are smooth projective varieties, we have the following isomorphism :
		\begin{equation}\label{cycle}
			Pic(B_j)\xrightarrow{\simeq} A_{(dim_{\bc}{B_j})-1}({B_j})\xrightarrow[cl_{B_j}]{\simeq} H_{2{(dim_{\bc}{B_j})-2}}({B_j})\xrightarrow{\simeq}H^2(B_j)
		\end{equation}
		\noindent
		where, the first isomorphism comes from (Example 2.1.1, \cite{ful}) and the last isomorphism is due to the well known Poincare duality. Hence,  $c_1: Pic(B_j)\to H^2(B_j)$ is an isomorphism by \eqref{cycle}.\\
		\noindent
		Now, using the result (Remark 21.18, \cite{bt}) on the cohomology ring
		of the induced flag bundle and an induction on the stages of
		$B_\bullet$, we see that $H^2(B_j)$ is generated by the first Chern
		classes of line bundles
		$$\{\mathcal{L}_{j,k}\,|\, 1\leq k\leq n_j+1\}\cup\,
		\bigcup_{l=1}^{j-1}\{p_j^*\,\circ \, \cdots \, \circ\,
		p^*_{l+1}(\mathcal{L}_{l,k}) \,|\, 1\leq k\leq n_l+1\}$$ for each
		$1\leq j\leq m$. Therefore, any cohomology class of degree 2 can be
		written as the first Chern class of a tensor product of the above line
		bundles. Hence the result follows.
	\end{proof}

	\begin{rema}{(Description of the flag Bott manifold $B_j$ using compact Lie groups)}\label{cfbm}
		We consider the orbit space $$\prod_{l=1}^{j} U(n_l+1)/\prod_{l=1}^{j}
		T(n_l+1)$$ for compact unitary groups $U(n_l+1)$ along with compact
		maximal torus \[T(n_l+1)\simeq (S^1)^{n_l+1}\] for each $1\leq l\leq
		j$. The right action is similar to \eqref{ract} :\\
		$((g_1,g_2,\ldots,g_j), (t_1,t_2,\ldots,t_j)):= (g_1t_1,
		{(\Psi_1^{(2)}(t_1))}^{-1}g_2t_2,\,{(\Psi_1^{(3)}(t_1))}^{-1}
		\\{(\Psi_2^{(3)}(t_2))}^{-1}g_3t_3,\ldots,{(\Psi_1^{(j)}(t_1))}^{-1}{(\Psi_2^{(j)}(t_2))}^{-1}\cdots
		{(\Psi_{j-1}^{(j)}(t_{j-1}))}^{-1}g_jt_j)$ where
		\[\displaystyle (g_1,g_2,\ldots,g_j)\in \prod_{l=1}^{j} U(n_l+1)\] and
		\[\displaystyle (t_1,t_2,\ldots,t_j)\in \prod_{l=1}^{j}
		T(n_l+1).\]
		Then the above manifold is a compact manifold which is diffeomorphic to $B_m$ since $U(n+1)/T(n+1)$ is diffeomorphic to the flag manifold \\$Flag(n+1)=GL(n+1)/B_{GL(n+1)}$ for each $n$.\hfill\qedsymbol
	\end{rema}

	\section{Flag Bott-Samelson Varieties}\label{fbs}
	In this section, we recall the definition of flag Bott-Samelson
	varieties introduced in (\cite[section 2.1]{fls}). In \cite{j}, flag
	Bott-Samelson varieties are considered in the more general setting of
	iterated fibrations of Schubert varieties without explicitly naming
	them. We also recall the one-parameter family of complex structures on
	the flag Bott-Samelson variety and its relation with flag Bott tower
	from \cite[section 4]{fls} in Theorem \ref{fbs2fbm}.
	
	\subsection{Definition of flag Bott-Samelson varieties}
	
	Let $G$ be a simply connected, semisimple algebraic group of rank $n$
	over $\bc$. Let $B\subset G$ be a Borel subgroup and $T\subseteq B$ be
	a maximal torus. Let
	$\displaystyle
	\mathfrak{g}=\mathfrak{h}\oplus\sum_{\alpha}\mathfrak{g}_\alpha$ be
	the Cartan decomposition of the Lie algebra $\mathfrak{g}$ into root
	spaces where $\mathfrak{h}:=Lie(T)$. Let $\Phi\subset \mathfrak{h}^*$
	denote the roots of $G$ and $\Phi^+\subset \Phi$ be a set of positive
	roots (corresponding to $B$), and
	$\Delta=\{\alpha_1,\ldots,\alpha_n\}\subset\,\Phi^+$ denote the set of
	simple roots. Let $\{\alpha_1^\vee,\ldots,\alpha_n^\vee\}$ denote the
	coroots and
	$\{{\omega}_1,\ldots, {\omega}_n\}\subseteq \mathfrak{h}^*$ the
	fundamental weights which are characterized by the relation
	$\langle {\omega}_i\,,\,\alpha_j^\vee\rangle=\delta_{ij}$ where
	$\delta_{ij}$ is the Kronecker symbol.  Similarly let
	$\{{\omega}^{\vee}_1,\ldots, {\omega}^{\vee}_n\}\subseteq
	\mathfrak{h}$ denote the fundamental coweights dual to the simple
	roots.

	Let $W$ denote the Weyl group of $G$. Let $s_i$ denote the simple
	reflection in $W$ corresponding to the simple root $\alpha_i$.  For a
	subset $I\subset [n]:=\{1,\,2,\ldots,n\}$, we define the subgroup
	$W_I:=\langle s_i\,|\,i\in I\rangle$ of $W$. In particular,
	$W_{\emptyset}=\{1\}$ and $W_{[n]}=W$.  We define, the parabolic
	subgroup
	$\displaystyle P_I:=\bigcup_{w\in
		W_I}BwB\,=\,\overline{Bw_IB}\,\subset G$ where $w_I$ denotes the
	longest element of $W_I$.
	
	\begin{defe}[Flag Bott-Samelson variety]
		Let $\mathscr{I}=(I_1,\ldots,I_r)$ be a sequence of subsets of $[n]$
		and let ${\bf{P}}_\mathscr{I}=P_{I_1}\times
		\cdots\times\,P_{I_r}$. We define a right action
		$\Theta: {\bf{P}}_\mathscr{I}\, \times B^r\to\,
		{\bf{P}}_\mathscr{I}$ given by
		\begin{equation}\label{bsact}
			\Theta\big((\fp_1,\ldots,\fp_r)\,,\,(b_1,\ldots,b_r)\big)=(\fp_1b_1,\, b_1^{-1}\fp_2b_2, ,\ldots,b_{r-1}^{-1}\fp_rb_r)
		\end{equation} for $(\fp_1,\ldots,\fp_r)\in {\bf{P}}_\mathscr{I}$ and $(b_1,\ldots,b_r)\in B^r:=\underbrace{B\times \cdots\times B}_r$.
		The flag Bott-Samelson variety $F_\mathscr{I}$ is defined to be the orbit space $$F_\mathscr{I}:={\bf{P}}_\mathscr{I}/\Theta.$$
	\end{defe}
	
	\begin{rema}
		If we take $\mathscr{I}=([n])$. Then we have
		${\bf{P}}_\mathscr{I}=G$.
		Therefore, the flag Bott-Samelson variety $F_\mathscr{I}$ is the flag variety $G/B$.\\
		Moreover, the flag Bott-Samelson variety is a Bott-Samelson variety
		if each $|I_k|=1$. Recall that a Bott-Samelson variety has a family
		of complex structure which induces a toric degeneration \cite{gk,p}.
	\end{rema}

	\begin{rema}
		For the subsequence $\mathscr{I'}=(I_1,\ldots,I_{r-1})$ of $\mathscr{I}$, there is a fibration structure on the flag Bott-Samelson variety $F_\mathscr{I}$ :
		\begin{equation}
			P_{I_r}/B \hookrightarrow F_\mathscr{I}\to F_\mathscr{I'}
		\end{equation} where the projection map $\pi_r: F_\mathscr{I}\to F_\mathscr{I'}$ is defined as $$[\fp_1,\ldots, \fp_{r-1},\,\fp_r]\mapsto [\fp_1,\ldots,\fp_{r-1}].$$
		One can also represent  $F_\mathscr{I}$ as $P_{I_1}\times _B F_\mathscr{I''}$, where $\mathscr{I''}=\,(I_2,\ldots,I_r)$.
	\end{rema}
	
	\subsection{Line bundles over the flag Bott-Samelson variety}
	Let $\mathscr{I}$ be a sequence of subsets of $[n]$. An integral
	weight
	\[\chi\in \mathbb{Z}\omega_1\bigoplus\cdots \bigoplus
	\mathbb{Z}\omega_n={X}^*(T)\] induces a homomorphism
	$e^{\chi}:B\lra \mathbb{C}^*$ by composing with the canonical map
	$\Gamma:B\lra T$.
	
	For $\chi_1,\ldots, \chi_r\in X^*(T)$ we can define the one
	dimensional complex representation $\mathbb{C}_{\chi_1,\ldots,\chi_r}$
	of $B^r$ where $B^r$ acts on $\mathbb{C}$ as follows:
	\[ (b_1,\ldots, b_r)\cdot v:= e^{\chi_1}(b_1)\cdots
	e^{\chi_r}(b_r)\cdot v .\]
	
	Let
	$\mathcal{L}_{\mathscr{I},\chi_1,\ldots,
		\chi_r}={\bf{P}}_{\mathscr{I}}\times_{B^r}
	\mathbb{C}_{\chi_1,\ldots,\chi_r} $ denote the associated line
	bundle on $F_\mathscr{I}={\bf{P}}_{\mathscr{I}}/B^r$.
	
	We let
	\[\mathcal{L}_{\mathscr{I},\chi}:=\mathcal{L}_{\mathscr{I},0,\ldots,0,\chi}.\]

	\subsection{Complex structures on flag Bott-Samelson variety
		$F_\mathscr{I}$}
	
	Let $\displaystyle\lambda=\sum_{i=1}^n a_i\cdot \omega_i^{\vee}$ for
	$a_i>0$ for all $1\leq i\leq n$. Then $\lambda$ defines a one
	parameter subgroup $\mathbb{C}^*\lra T$ so that
	$\alpha_i\circ \lambda(t)=t^{a_i}$ for $t\in \mathbb{C}^*$ and
	$1\leq i\leq n$.  Then $\langle \lambda,\alpha_i\rangle=a_i>0$ for
	every $1\leq i\leq n$.  We can choose $a_i=a>0$ for all $i$. We define
	$\Gamma_t:B\lra B$ by
	$\Gamma_t(b)=\lambda(t)\cdot b\cdot \lambda(t)^{-1}$ for
	$t\in \mathbb{C}^*$. We know by \cite[Proposition 3.5]{gk} that
	$\displaystyle\Gamma=\lim_{t\ra 0}\Gamma_{t}$ where $\Gamma: B\lra T$
	is the canonical surjection. We let $\Gamma_0:=\Gamma$.

	Using $\Gamma_t$, we define a right action
	$\Theta_t : {\bf{P}}_\mathscr{I}\,\times\, B^r\to
	{\bf{P}}_\mathscr{I}$ as
	\begin{equation}\label{diffac}
		\Theta_t\big((\fp_1,\ldots,\fp_r)\,,\,(b_1,\ldots,b_r)\big)=(\fp_1b_1,\, \Gamma_t(b_1)^{-1}\fp_2b_2, \ldots,\Gamma_t(b_{r-1})^{-1}\fp_rb_r)
	\end{equation} for $(\fp_1,\ldots,\fp_r)\in {\bf{P}}_\mathscr{I}$ and $(b_1,\ldots,b_r)\in B^r$.  Note that, $\Theta_1$ is same as the right action $\Theta$ in \eqref{bsact}. We consider the family of orbit spaces $$F_\mathscr{I}^t:={\bf{P}}_\mathscr{I}/\Theta_t$$ under the right action $\Theta_t$ for $t\in \bc$.

	Let
	$\mathcal{L}^{t}_{\mathscr{I},\chi_1,\ldots,
		\chi_r}={\bf{P}}_{\mathscr{I}}\times_{B^r}
	\mathbb{C}_{\chi_1,\ldots,\chi_r} $ denote the line bundle associated
	to the $1$-dimensional representation
	$\mathbb{C}_{\chi_1,\ldots,\chi_r}$ of $B^r$ on
	$F^{t}_\mathscr{I}={\bf{P}}_{\mathscr{I}}/B^r$ where the action of
	$B^r$ on ${\bf{P}}_{\mathscr{I}}$ is via $\Theta_t$. We let
	\[\mathcal{L}^t_{\mathscr{I},\chi}:=\mathcal{L}^t_{\mathscr{I},0,\ldots,0,\chi}
	.\]

	\begin{propo}\label{diffeo}(see \cite[Proposition 4.6]{fls})
		Let $\mathscr{I}=(I_1,\ldots,I_r)$ be a sequence of subsets of
		$[n]$. Then $\{F_\mathscr{I}^t\}_{t\in \bc}$ are all
		diffeomorphic. Moreover, $F_\mathscr{I}^1=F_\mathscr{I}$.
	\end{propo}

	Henceforth we shall consider $\mathscr{I}=(I_1,\ldots,I_r)$ be a
	sequence of subsets of $[n]$ such that the Levi subgroup $L_{I_k}$ of
	the parabolic subgroup $P_{I_k}$ has Lie type $A_{n_k}$ for all
	$1\leq k\leq r$. We shall take an enumeration
	$I_k=\{u_{k,1},\ldots,u_{k,n_k}\}$ which satisfies
	\begin{equation}\label{con}
		\langle\alpha_{u_{k,s}},\, \alpha_{u_{k,t}}^\vee\rangle=  \begin{cases}
			2 & \text{if $s=t$,}\\
			-1 &\text{if $s-t=\pm 1$,}\\
			0 &\text{otherwise}
		\end{cases}
	\end{equation}

	\bth\label{limitflagbott}(see \cite[Prop. 4.8]{fls}) Let
	$F_\mathscr{I}$ be a flag Bott-Samelson variety. Let $\mathscr{I}'$
	denote the subsequence $(I_1,\ldots, I_{r-1})$ of $\mathscr{I}$. Then
	$F_\mathscr{I}^0$ is diffeomorphic to the following flag bundle over
	$F_{\mathscr{I}'}^0$
	\[ F_\mathscr{I}^0\simeq
	Flag(\mathcal{L}^0_{\mathscr{I}',\chi_1}\oplus \cdots \oplus
	\mathcal{L}^0_{\mathscr{I}',\chi_{n_r}}\oplus 1_{\bc}) \] where
	$\chi_j=\alpha_{u_{r,j}}+\cdots+\alpha_{u_{r,n_r}}\in \mathfrak{h}^*$
	for $1\leq j\leq n_r$,
	$\mathcal{L}^0_{\mathscr{I},\chi}:=\mathcal{L}^0_{\mathscr{I},0,\ldots,0,\chi}$
	and $1_{\bc}$ denotes the trivial complex line bundle.  \eeth

	Moreover, $\mathcal{L}^{0}_{\mathscr{I},\chi_1,\ldots,\chi_r}$ is
	isomorphic to the line bundle on the flag Bott manifold
	$F^{0}_\mathscr{I}$ corresponding to the vector
	$\displaystyle ( \mathbf{a}_1,\ldots, \mathbf{a}_r)\in \prod_{k=1}^r
	\mathbb{Z}^{n_k+1}$ where
	$\mathbf{a}_{k}=(a_{k}(1),\ldots, a_{k}(n_k+1))\in \mathbb{Z}^{n_k+1}$ is
	defined by
	\[a_{k}(l)=\langle \chi_k+\cdots+\chi_r,
	\alpha^{\vee}_{u_{k,l}}+\cdots+\alpha^{\vee}_{u_{k,n_k}} \rangle\] for
	$1\leq l\leq n_k$ and $a_{k}(n_k+1)=0$. Indeed,
	$\displaystyle (\mathbf{a}_1,\ldots, \mathbf{a}_r)\in \prod_{k=1}^r
	\mathbb{Z}^{n_k+1}$ can be identified with the first Chern class of
	the line bundle $\mathcal{L}^{0}_{\mathscr{I},\chi_1,\ldots,\chi_r}$.
	
	Recall by Lemma \ref{line2} that on the flag Bott manifold
	$F^{0}_\mathscr{I}$ the tautological line bundles
	$\mathcal{L}^0_{r,k}$ for $1\leq k\leq n_r+1$ correspond to the vector bundle
	$\eta({\bf 0},\ldots, {\bf 0}, {\bf e}_k)$, where
	${\bf e}_k\in \mathbb{Z}^{n_r+1}$ is the vector with $1$ at the $k$th
	position and $0$ elsewhere. Thus $\mathcal{L}^0_{r,k}$ are associated
	to the characters $\epsilon_1,\ldots, \epsilon_{n_r+1}$ of the maximal
	torus $T_{SL_{n_r+1}}$ of $SL_{n_r+1}$ corresponding to the $n_r+1$
	coordinate projections. Now, $T_{SL_{n_r+1}}$ is of rank $n_r$ and
	$Lie(T_{SL_{n_r+1}})$ is free $\mathbb{Z}$-module on the fundamental
	weights $\omega_{u_{r,1}},\ldots, \omega_{u_{r,n_r}}$. Further,
	$\epsilon_1=\omega_{u_{r,1}},
	\epsilon_{2}=\omega_{u_{r,2}}-\omega_{u_{r,1}},
	\ldots,\epsilon_{n_r}=\omega_{u_{r,n_r}}-\omega_{u_{r,n_r-1}}$ and
	$\epsilon_{n_r+1}=-\omega_{u_{r,n_r}}$.
	
	Thus we have the isomorphisms
	$\mathcal{L}^0_{r,1}\simeq
	\mathcal{L}^0_{\mathscr{I},\omega_{u_{r,1}}},
	\mathcal{L}^0_{r,2}\simeq
	\mathcal{L}^0_{\mathscr{I},{\omega_{u_{r,2}}-\omega_{u_{r,1}}}},\ldots,\\
	\mathcal{L}^0_{r,n_r}\simeq \mathcal{L}^0_{\mathscr{I},
		\omega_{u_{r,n_r}}-\omega_{u_{r,n_r-1}}}$ and
	$\mathcal{L}^0_{r,n_r+1}\simeq \mathcal{L}^0_{\mathscr{I},
		-\omega_{u_{r,n_r}}}$ on $F^{0}_\mathscr{I}$.
	
	\brem\label{corresalgtop} Under the diffeomorphism of
	$F^{0}_\mathscr{I}$ and $F^{1}_\mathscr{I}$, the line bundle
	$\mathcal{L}_{\mathscr{I},\chi_1,\ldots, \chi_r}$ corresponds to
	$\mathcal{L}^0_{\mathscr{I},\chi_1,\ldots, \chi_r}$. In particular,
	$\mathcal{L}^0_{r,1}$ corresponds to
	$\mathcal{L}_{\mathscr{I},\omega_{u_{r,1}}}$, $\mathcal{L}^0_{r,k}$
	corresponds to
	$\mathcal{L}_{\mathscr{I},\omega_{u_{r,k}}-\omega_{u_{r,k-1}}}$ for
	$2\leq k\leq n_r$ and $\mathcal{L}^0_{r,n_r+1}$ corresponds to
	$\mathcal{L}_{\mathscr{I},-\omega_{u_{r,n_r}}}$. Thus under the
	diffeomorphism given by the deformation of complex structures, the
	tautological line bundles $\mathcal{L}^0_{r,k}$ for every
	$1\leq k\leq n_r+1$ correspond to the topological restrictions of
	certain algebraic line bundles on the flag Bott-Samelson variety
	$F_\mathscr{I}$. Similarly for every subsequence
	$\mathscr{I}'=(I_1,\ldots, I_j)$ for $1\leq j\leq r$, the tautological
	line bundles $\mathcal{L}_{j,k}$ for $1\leq k\leq n_j+1$ on the
	$j$-stage flag Bott manifold $F^0_{\mathscr{I}'}$ correspond
	respectively to the restrictions of the algebraic line bundles
	$\mathcal{L}_{\mathscr{I},\omega_{u_{j,1}}}$,
	$\mathcal{L}_{\mathscr{I},\omega_{u_{j,k}}-\omega_{u_{j,k-1}}}$ for
	$2\leq k\leq n_j$ and $\mathcal{L}_{\mathscr{I},-\omega_{u_{j,n_j}}}$
	on the flag Bott-Samelson variety $F_{\mathscr{I}'}$.  \erem

	\begin{guess}\label{fbs2fbm}(\cite[Theorem 4.10]{fls})
		The manifold $F_\mathscr{I}^0$ is an $r$-stage flag Bott manifold
		which is determined by a sequence of
		matrices
		$$\mathfrak{M}:=(Q_l^{(j)})_{1\leq l<j\leq r}\in \prod_{1\leq
			l<j\leq r} M_{n_j+1,n_l+1}(\bz)$$ in the sense of Definition
		\ref{defquo}, where $Q_l^{(j)}(p,q)$
		is
		$$\langle\alpha_{u_{j,p}}+\ldots+ \alpha_{u_{j,n_j}}\,,\,
		\alpha^\vee_{u_{l,q}}+\ldots+ \alpha^\vee_{u_{l,n_l}}\rangle$$ if
		$1\leq p\leq n_j$ and $1\leq q\leq n_l$, and $0$ otherwise.
		
	\end{guess}
	\begin{exam}
		let $G=SL(4)$. Consider the sequence $\mathscr{I}=(\{1,2\}, \{1,2\})$. Hence $u_{1,1}=1,\, u_{1,2}=2,\, u_{2,1}=1,\, u_{2,2}=2$. Then the manifold  $F_\mathscr{I}^0$ is a $2$-stage flag Bott manifold which is determined by a matrix $$Q^{(2)}_1:=\begin{bmatrix} \langle\alpha_1+\alpha_2, \alpha_1^\vee+\alpha_2^\vee\rangle &  \langle\alpha_1+\alpha_2, \alpha_2^\vee\rangle & 0\\
			\langle\alpha_2, \alpha_1^\vee+\alpha_2^\vee\rangle &  \langle\alpha_2, \alpha_2^\vee\rangle & 0\\  0&0&0\end{bmatrix}= 
		\begin{bmatrix} 2 &  1 & 0\\1 & 2 & 0\\ 0&0&0\end{bmatrix}$$
	\end{exam}

	\begin{coro}\label{bs}
		Suppose that the flag Bott-Samelson variety is a Bott-Samelson
		variety ~i.e~ $|I_k|=1~~\forall~~ k=1,2,...,r$, so that
		$n_1=n_2=\cdots=n_r=1$.  Hence from Theorem
		\ref{fbs2fbm},
		$$Q_l^{(j)}=\begin{bmatrix}
			\langle\alpha_{u_{j,1}}\,,\,\alpha^\vee_{u_{l,1}}\rangle & 0\\ 0 &
			0
		\end{bmatrix} \text{ for each $1\leq l<j\leq r$}$$ (see
		\cite[Section 3.7]{gk}).
	\end{coro}

	\section{Topological K-theory of Flag Bott manifolds}\label{kth}
	
	Let $E$ be a complex vector bundle of rank $n$ over a compact manifold
	$X$. Then the exterior power $\Lambda^i(E)$ is also a complex vector
	bundle over $X$ with fiber $\Lambda^i(E)_x=\Lambda^i(E_x)$ for each
	$x\in X$. Moreover, if $E$ is a direct sum of $n$ complex line bundles
	$E\simeq L_1\oplus L_2\oplus \cdots \oplus L_n$ over $X$, then we have
	the isomorphisms
	$\Lambda^1(E)\simeq L_1\oplus L_2\oplus \cdots \oplus L_n$,
	$\Lambda^2(E) \simeq \bigoplus_{i<j}L_i\otimes L_j$ etc. In
	particular, for each $\displaystyle 1\leq k\leq n$
	$$\Lambda^k(E)\simeq\bigoplus_{i_1<i_2<\ldots<i_k}L_{i_1}\otimes
	L_{i_2}\otimes \cdots \otimes L_{i_k}$$ and the class $[\Lambda^k(E)]$
	in $K(X)$ can be interpreted as $e_k([L_1]\,,[L_2],\ldots,[L_n])$
	where $e_k$ denotes the $k$-th elementary symmetric polynomial.
	
	We recall from the definition of a flag bundle $\pi : Flag(E)\to X$
	associated to an $n$ dimensional complex vector bundle $E$ over a
	compact manifold $X$, that we have a sequence of canonical
	sub-bundles $0\subset E_1\subset E_2\subset\cdots\subset E_n=\pi^*(E)$
	such that the successive quotients $E_i/E_{i-1}$ are well defined line
	bundles over $Flag(E)$ with
	$\displaystyle\bigoplus_{i=1}^n(E_i/E_{i-1})\simeq \pi^*E$. We shall
	denote the class of $(E_i/E_{i-1})$ in $K^*(Flag(E))$ by $h_i$.
	
	We recall the following classical theorem describing $K^*(Flag(E))$ as
	$K^*(X)$-algebra.

	\begin{guess}\cite[Chapter IV : Theorem 3.6]{k}\label{kar}
		Let
		$$\phi : K^*(X)[x^{\pm 1}_1,\ldots, x^{\pm 1}_n]\to K^*(Flag(E))$$
		be the $K^*(X)$- algebra homomorphism sending $x_i$ to $h_i$. Then
		$\phi$ is surjective and its kernel is the ideal $I$ in the Laurent
		polynomial ring $K^*(X)[x^{\pm 1}_1,\ldots, x^{\pm1}_n]$ generated
		by the elements \be\label{elementary} e_i(x_1,\ldots,
		x_n)-[\Lambda^i(E)],\ee where $e_i$ is the $i$-th elementary
		symmetric polynomial in the $x_i$'s. Hence, $\phi$ induces an
		isomorphism
		$$K^*(X)[x^{\pm1}_1,\ldots, x^{\pm1}_n]/I\simeq K^*(Flag(E))$$
	\end{guess}
	
	\brem In \cite{k} the ring $K^*(Flag(E))$ is expressed as the quotient
	of the polynomial ring $K^*(X)[x_1,\ldots, x_n]$ by the ideal $I'$
	generated by the elements \eqref{elementary}. Now, by \cite[Section
	2.6]{athir}, we note that $(1-[L_i])$ are nilpotent elements in
	$K^*(Flag(E))$ (since $X$ is finite dimensional being a compact
	manifold, the total space $Flag(E)$ is also finite dimensional). It
	follows that in $K^*(Flag(E))$, $[L_i]^{-1}$ can itself be expressed
	as a polynomial in $(1-[L_i])$. For this reason, we can as well
	express $K^*(Flag(E))$ as a quotient of the Laurent polynomial ring
	$K^*(X)[x^{\pm 1}_1,\ldots, x^{\pm1}_n]$ by the ideal $I$ generated by
	the elements \eqref{elementary}, where
	$K^*(X)[x^{\pm 1}_1,\ldots, x^{\pm1}_n]$ can be identified with the
	localization $K^*(X)[x_1,\ldots, x_n]_{(x_1\cdots x_n)}$ of
	$K^*(X)[x_1,\ldots, x_n]$ and $I$ is the extension of the ideal $I'$
	in $K^*(X)[x_1,\ldots, x_n]$. This is with a view to apply the Theorem
	\ref{kar} iteratively for the flag Bott manifolds. For, when the base
	space $X$ itself is a flag manifold, $[\Lambda^i(E)]$ may involve
	classes of negative powers of the tautological line bundles. Thus
	considering $K^*(Flag(E))$ as a quotient of the Laurent polynomial
	ring at each stage will facilitate our arguments as well as simplify
	the presentation for the $K$-ring of flag Bott manifolds. \erem
	
	\subsection{$K$-ring of flag Bott manifolds}\label{topk}
	Recall from Proposition \ref{prop} that for every $j> 1$,
	$B_j^{quo}= Flag(\eta^{(j)})$ is a flag bundle over $B_{j-1}^{quo}$
	where
	\begin{equation}\label{a}
		\eta^{(j)}:=\bigoplus_{k=1}^{n_j+1}\eta({\bf{v}}^{(j)}_{k,1},{\bf{v}}^{(j)}_{k,2},\ldots,{\bf{v}}^{(j)}_{k,j-1})\end{equation} 
	is a direct sum of $n_j+1$ line bundles over $B_{j-1}^{quo}$ and
	${\bf{v}}^{(j)}_{k,l}$  is the $k$-th row vector of the matrix
	$P^{(j)}_l$ for each $1\leq l\leq j-1$ for the collection of matrices
	$$\mathfrak{P}:=(P_l^{(j)})_{1\leq l<j\leq
		m}\in \prod_{1\leq l<j\leq m}M_{n_j+1,n_l+1}(\bz).$$ 
	
	\bth\label{inductionkring} We have an $K^*(B_{j-1}^{quo})$- algebra
	isomorphism :
	$$\frac{K^*(B_{j-1}^{quo})\,[y^{\pm1}_{j,1},\,
		y^{\pm1}_{j,2},\ldots,\,y^{\pm1}_{j,n_j+1}]}{\mathcal {I}_j}\,\simeq\,K^*(Flag(\eta^{(j)}))=\,
	K^*(B_j^{quo})$$ given by
	$y_{j,k} + \mathcal{I}_j\mapsto [\mathcal{L}_{j,k}]$ where
	$\mathcal{I}_j$ is the ideal in the Laurent polynomial ring
	$K^*(B_{j-1}^{quo})\,[y^{\pm 1}_{j,1},\,
	y^{\pm1}_{j,2},\ldots,\,y^{\pm1}_{j,n_j+1}]$ generated by the elements
	$$e_r(y_{j,1},\, y_{j,2},\ldots,\,y_{j,n_j+1})-[\Lambda^r(\eta^{(j)})]
	;\text{ for each $1\leq r\leq n_j+1$}.$$\eeth \begin{proof} The theorem
		follows from Theorem \ref{kar}.\end{proof}

	Let ${\bf{y}}_{j}$ denote the collection of variables
	$\{y^{\pm1}_{j,1},\,y^{\pm 1}_{j,2},\ldots,y^{\pm1}_{j,n_j+1}\}$ for every $1\leq j\leq
	m$.
	
	Let $\mathcal{R}:=\bz[{\bf{y}}_{j}\,|\, 1\leq j\leq m]$ denote the
	Laurent polynomial ring in the variables
	$y^{\pm1}_{j,1},\,y^{\pm1}_{j,2},\ldots,y^{\pm1}_{j,n_j+1}$ for
	$1\leq j\leq m$.
	
	Let $e_{r}(\bf{y}_{j})\in \mathcal{R}$ denote the $r$th elementary
	symmetric function in $y_{j,1},\,y_{j,2},\ldots,y_{j,n_j+1}$.
	
	Let $P_l^{(j)}=(P_l^{(j)}(r,s))\in M_{n_j+1,n_l+1}(\bz)$ for
	$1\leq r\leq n_j+1$, $1\leq s\leq n_l+1$.
	
	Let $\mathcal{I}_1$ denote the ideal in $\mathcal{R}$ generated by the
	polynomials
	$$e_r({\bf{y}}_{1})- \begin{pmatrix} n_1+1\\r\end{pmatrix}, \text{ for
		every $1\leq r\leq n_1+1$}$$ and let $\mathcal{I}_j$ denote the
	ideal in $\mathcal{R}$ generated by the polynomials
	\[e_r({\bf{y}}_{j})-e_r\big(\prod_{s=1}^{j-1}(\prod_{i=1}^{n_s+1}y_{s,i}^{P^{(j)}_s(k,i)})\,|
	1\leq k\leq n_j+1\big)\] for every $2\leq j\leq m$ and
	$1\leq r\leq n_j+1$. 
	
	Let $\mathcal{I}:=\mathcal{I}_1+\cdots+\mathcal{I}_m$ which is the
	smallest ideal in $\mathcal{R}$ containing
	$\mathcal{I}_1,\ldots, \mathcal{I}_m$.
	
	\begin{guess}\label{mainth}
		Suppose that $B_\bullet =\{B_j\,|\,0\leq j\leq m\}$ is an $m$- stage
		flag Bott manifold of Lie type $A$, determined by a set of integer
		matrices
		$\mathfrak{P}:=(P_l^{(j)})_{1\leq l<j\leq m}\in \prod_{1\leq l<j\leq
			m}M_{n_j+1,n_l+1}(\bz)$. The map from $\mathcal{R}$ to $K^*(B_m)$
		which sends $y_{m,k}$ to $[\mathcal{L}_{m,k}]$ for
		$1\leq k\leq n_m+1$ and $y_{l,k}$ to
		\[[p_m^*\,\circ\, \cdots\,\circ\, p^*_{l+1}(\mathcal{L}_{l,k})]\] for
		$1\leq k\leq n_l+1$ and $1\leq l\leq m-1$,  induces
		an isomorphism $$\mathcal{R}/\mathcal{I}\simeq K^*(B_m).$$
	\end{guess}
	
	\begin{proof} We prove this by induction on $m$. The result is true
		for $m=1$ since $B_1$ is a flag manifold associated to direct sum of
		line bundles $\mathcal{L}_{1,k}$ for $1\leq k\leq n_1+1$. In this
		case we know have the classical presentation given
		by \[\mathbb{Z}[{\bf{y}}_{1}]/{\mathcal{I}_1}\] (see
		\cite[Proposition 2.7.13]{at}, \cite[Chapter IV : Theorem
		3.6]{k}). Now, we assume by induction that the result is true for
		$B_{m-1}$. Thus \be\label{isoind}K^*(B_{m-1})\simeq
		\mathbb{Z}[{\bf{y}}_{1},\ldots, {\bf{y}}_{m-1}]
		/{\mathcal{I}_1+\cdots+\mathcal{I}_{m-1}}\ee under the isomorphism
		which sends $y_{{m-1},k}$ to $[\mathcal{L}_{{m-1},k}]$ for
		$1\leq k\leq n_{m-1}+1$ and $y_{l,k}$ to
		$p_{m-1}^*\,\circ\, \cdots\,\circ\, p^*_{l+1}(\mathcal{L}_{l,k})$
		for $1\leq k\leq n_l+1$ and $1\leq l\leq m-2$.

		Now by Theorem \ref{inductionkring} and the induction hypothesis we
		get that $K^*(B_m)$ is isomorphic
		to
		\[\big ({\mathbb{Z}[{\bf{y}}_{1},\ldots,
			{\bf{y}}_{m-1}]/{\mathcal{I}_1+\cdots+\mathcal{I}_{m-1}}}\big)[{\bf{y}}_{m}]\] modulo the
		ideal generated by the elements
		\[e_r(y_{m,1},\, y_{m,2},\ldots,\,y_{m,n_m+1})-[\Lambda^r(\eta^{(m)})]
		;\text{ for each $1\leq r\leq n_m+1$}.\] The theorem will follow if
		we show that
		\[e_r\big(\prod_{s=1}^{m-1}(\prod_{i=1}^{n_s+1}y_{s,i}^{P^{(m)}_s(k,i)})
		| 1\leq k\leq n_m+1\big).\] maps to $[\Lambda^r(\eta^{(m)})]$ under
		the isomorphism \eqref{isoind}. Thus it suffices to show that
		$\prod_{s=1}^{m-1}(\prod_{i=1}^{n_s+1}y_{s,i}^{P^{(m)}_s(k,i)})$
		maps to $\eta({\bf v}^{(m)}_{k,1},\ldots, {\bf v}^{(m)}_{k, m-1})$ for
		$1\leq k\leq n_m+1$. Since $y_{m-1,k}$ maps to
		$[\mathcal{L}_{{m-1},k}]$ for $1\leq k\leq n_{m-1}+1$ and $y_{s,i}$
		maps to
		$p_{m-1}^*\,\circ\, \cdots\,\circ\, p^*_{s+1}(\mathcal{L}_{s,i})$ for
		$1\leq s\leq m-2$ and $1\leq i\leq n_s+1$, it further suffices to show
		that
		\[\eta({\bf v}^{(m)}_{k,1},\ldots, {\bf v}^{(m)}_{k, m-1})\] is isomorphic to
		\[\displaystyle \bigotimes_{s=1}^{m-2} \bigotimes_{i=1}^{n_s+1}
		(p_{m-1}^*\,\circ\, \cdots\,\circ\,
		p^*_{s+1}(\mathcal{L}_{s,i}))^{P^{(m)}_s(k,i)}\bigotimes_{i=1}^{n_{m-1}+1} \mathcal{L}_{m-1,i}^{P^{(m)}_{m-1}(k,i)} .\]

		We recall that $\mathcal{L}_{m-1, j}$ is isomorphic to the line bundle
		$\eta({\bf 0}, \ldots, {\bf 0}, {\bf e}_j)$ where
		${\bf e}_j\in \mathbb{Z}^{n_{m-1}+1}$ has $1$ at the $j$-th place and
		$0$ everywhere else and
		$p_{m-1}^*\,\circ\, \cdots\,\circ\, p^*_{s+1} (\mathcal{L}_{s,i})$ is
		isomorphic to the line bundle
		$\eta({\bf 0},\ldots, {\bf e}_i,\ldots, {\bf 0})$ where
		${\bf e}_i\in \mathbb{Z}^{n_s+1}$ has $1$ at the $i$th place and $0$
		everywhere else. Also ${\bf v}^{(m)}_{k,s}\in \mathbb{Z}^{n_s+1}$ is the
		$k$th row vector of the matrix $P^{(m)}_{s}$. Thus
		${\bf v}^{(m)}_{k,s}=\sum_{i=1}^{n_s+1} P^{(m)}_s(k,i)\cdot {\bf
			e}_i$. This implies that
		\[({\bf v}^{(m)}_{k,s})_{s=1}^{m-1}=(\sum_{i=1}^{n_s+1}
		P^{(m)}_s(k,i)\cdot {\bf e}_i )_{s=1}^{m-1}.\] Now, the claim
		follows by repeated application of Lemma \ref{tensor}.  \end{proof}

	\begin{exam}
		We now determine the presentation of the $K$ ring for the Example \ref{ex} in terms of generators and relations.\\
		\noindent
		$K^*(B_0)\cong\bz$\\
		$K^*(B_1)\cong \bz[y_{1,1}^{\pm 1},\, y_{1,2}^{\pm
			1},\,y_{1,3}^{\pm 1}]\,/\,{\mathcal{I}_1}$.\\
		$K^*(B_2)\cong\bz[y_{1,1}^{\pm 1},\, y_{1,2}^{\pm 1},\,y_{1,3}^{\pm 1},\, y_{2,1}^{\pm 1},\,y_{2,2}^{\pm 1}]\,/\,\mathcal{I}_1 + \mathcal{I}_2$\\
		$K^*(B_3)\cong\bz[y_{1,1}^{\pm 1},\, y_{1,2}^{\pm 1},\,y_{1,3}^{\pm 1}, \,y_{2,1}^{\pm 1},\,y_{2,2}^{\pm 1},\, y_{3,1}^{\pm 1},\, y_{3,2}^{\pm 1}]/\mathcal{I}_1+\mathcal{I}_2+\mathcal{I}_3$\\
		\noindent
		where, $\mathcal{I}_1$ is generated by the elements $$e_r(y_{1,1},\, y_{1,2},\,y_{1,3})- \begin{pmatrix} 3\\r\end{pmatrix}, \text{ for each $1\leq r\leq 3$.}$$
		
		$\mathcal{I}_2$ is generated by the elements $$e_r(y_{2,1},\,y_{2,2})-e_r(y_{1,1}^{a_1}\,y_{1,2}^{a_2}\, y_{1,3}^{a_3}, \, y_{1,1}^{b_1}\, y_{1,2}^{b_2}\, y_{1,3}^{b_3})$$ for each $1\leq r\leq 2$.
		
		$\mathcal{I}_3$ is generated by the elements $$e_r(y_{3,1},\,y_{3,2})-e_r(y_{1,1}^{c_1}\,y_{1,2}^{c_2}\, y_{1,3}^{c_3}\,y_{2,1}^{f_1}\,y_{2,2}^{f_2}\,, \, y_{1,1}^{d_1}\, y_{1,2}^{d_2}\, y_{1,3}^{d_3}\,y_{2,1}^{0}\,y_{2,2}^{0})$$  for each $1\leq r\leq 2$.
	\end{exam}

	\subsection{$K$-ring of flag Bott-Samelson
		varieties}\label{gr}
	
	Suppose that  $\mathscr{I}=(I_1,\ldots,I_r)$ be a sequence of subsets
	of $[n]$ such that the Levi subgroup $L_{I_k}$ of the parabolic
	subgroup $P_{I_k}$ has Lie type $A_{n_k}$ for all $1\leq k\leq
	r$. Take an enumeration $I_k=\{u_{k,1},\ldots,u_{k,n_k}\}$ which
	satisfies the conditions \eqref{con}.

	Let $\mathscr{K}^*(F_\mathscr{I})$ denote the algebraic $K$-ring of
	the flag Bott-Samelson variety $F_\mathscr{I}$.  Let
	$\mathcal{R}':=\bz[{\bf{y}}_{j}\,|\, 1\leq j\leq r]$ where
	${\bf{y}}_{j}$ denotes the collection of variables
	$\{y^{\pm 1}_{j,1},\,y^{\pm 1}_{j,2},\ldots,y^{\pm 1}_{j,n_j+1}\}$ for every
	$1\leq j\leq r$. Let $\mathcal{I}_1$ be the ideal in $\mathcal{R}'$
	generated by the polynomials $\displaystyle e_t({\bf{y}}_{1})- \begin{pmatrix} n_1+1\\t\end{pmatrix}$ for each $1\leq t\leq n_1+1$.
	And for each $2\leq j\leq r$, let $\mathcal{I}_j$ be the ideal in $\mathcal{R}'$
	generated by the polynomials
	\[e_t({\bf{y}}_{j})-e_t\big(\prod_{s=1}^{j-1}(\prod_{i=1}^{n_s+1}y_{s,i}^{Q^{(j)}_s(k,i)})\,|
	1\leq k\leq n_j+1\big) ~\mbox{for each}~ 1\leq t\leq n_j+1,\] where
	$\displaystyle Q_l^{(j)}(p,q)=\langle \alpha_{u_{j,p}}+\,.\,.\,.+
	\alpha_{u_{j,n_j}}\,,\, \alpha^\vee_{u_{l,q}}+\ldots+
	\alpha^\vee_{u_{l,n_l}}\rangle$ if $1\leq p\leq n_j$ and
	$1\leq q\leq n_l$, and $0$ otherwise. Let $\mathcal{I}'$ denote the
	ideal $\mathcal{I}_1+\cdots+\mathcal{I}_r$ in $\mathcal{R}'$

	\brem\label{higher K} Since the flag manifold has algebraic cell
	decomposition given by the Schubert cells the iterated flag bundles
	have cellular structure with cells only in even dimension. Thus it
	follows by \cite[Section 2.5]{athir} that $K^*(B_m)=K^0(B_m)$ i.e.
	$K^{-1}(B_m)=0$. This can alternately be seen by Leray Hirch theorem
	for $K$-theory (see \cite[Theorem 1.3, Chapter IV]{k}) or by the
	degeneracy of the Atiyah-Hirzebruch spectral sequence \cite[Section
	2.4]{athir}. Furthermore, since $F_\mathscr{I}$ and $F^0_\mathscr{I}$
	are diffeomorphic, ${K}^*(F_\mathscr{I})\simeq
	K^*(F^0_\mathscr{I})$. Thus it follows that
	${K}^*(F_\mathscr{I})=K^0(F_\mathscr{I})$.  \erem

	As an immediate corollary of Theorem \ref{mainth}, we have the
	following presentation of the topological $K^0$- ring (and hence the
	Grothendieck ring $\mathscr{K}^0$) of $F_\mathscr{I}$ in terms of
	generators and relations.
	
	\begin{coro}\label{kfbs} The map $\mathcal{R}'\lra
		\mathscr{K}^0(F_\mathscr{I})$ which sends $y_{j,k}$ to
		\[[p_{r}^*\,\circ\, \cdots\,\circ\,
		p^*_{j+1}(\mathcal{L}_{\mathscr{I},\omega_{u_{j,k}}-\omega_{u_{j,k-1}}})]~\mbox{for}~2\leq
		k\leq n_j,\]  $y_{j,1}$ to
		\[[p_{r}^*\,\circ\, \cdots\,\circ\,
		p^*_{j+1} (\mathcal{L}_{\mathscr{I},\omega_{u_{j,1}}})]~\mbox{and}\] 
		$y_{j,n_j+1}$ to
		\[[p_{r}^*\,\circ\, \cdots\,\circ\,
		p^*_{j+1}(\mathcal{L}_{\mathscr{I},-\omega_{u_{j,n_j}}})]\] for
		$1\leq j\leq r-1$, and $y_{r,k}$ to
		\[[\mathcal{L}_{\mathscr{I},\omega_{u_{r,k}}-\omega_{u_{r,k-1}}}]~\mbox{for}~2\leq
		k\leq n_r,\]  $y_{r,1}$ to
		\[[\mathcal{L}_{\mathscr{I},\omega_{u_{r,1}}}]~\mbox{and}\] 
		$y_{r,n_r+1}$ to
		\[[\mathcal{L}_{\mathscr{I},-\omega_{u_{r,n_r}}}]\] induces an
		isomorphism \be\label{iso}\mathscr{K}^0(F_\mathscr{I}) \simeq
		K^0(F_\mathscr{I})\simeq \mathcal{R}'/\mathcal{I}'.\ee
	\end{coro}
	
	\begin{proof}
		From Proposition \ref{diffeo}, $F_\mathscr{I}^0$ and
		$F_\mathscr{I}^1\,=\, F_\mathscr{I}$ are diffeomorphic. Hence,
		\be\label{e1} K^0(F_\mathscr{I}) \simeq K^0(F_\mathscr{I}^0).\ee
		From Theorem \ref{fbs2fbm}, the manifold $F_\mathscr{I}^0$ is an
		$r$-stage flag Bott manifold which is determined by a sequence of
		matrices
		$\displaystyle\mathfrak{M}:=(Q_l^{(j)})_{1\leq l<j\leq r}\in \prod_{1\leq l<j\leq
			r} M_{n_j+1,n_l+1}(\bz)$ in the sense of Definition
		\ref{defquo}. Hence, from Theorem \ref{mainth}, the map from
		$\mathcal{R}'$ to $K^0(F_\mathscr{I}^0)$ which takes $y_{j,k}$ to
		$[p_{r}^*\,\circ\, \cdots\,\circ\, p^*_{j+1}(\mathcal{L}_{j,k})]$
		for $1\leq j\leq r-1$ and $y_{r,k}$ to $[\mathcal{L}_{r,k}]$ induces
		an isomorphism of $\mathbb{Z}$-algebras \be\label{e2}
		\mathcal{R}'/\mathcal{I}'\simeq K^0(F_\mathscr{I}^0).\ee Moreover,
		by \cite[Lemma 4.2]{su} and Remark \ref{corresalgtop} it follows
		that the forgetful map \be\label{e3}f:\mathscr{K}^0(F_\mathscr{I})
		\to K^0(F_\mathscr{I})\ee is an isomorphism of rings where
		$f([\mathcal{L}_{\mathscr{I},{\omega_{u_{j,k}}}-\omega_{u_{j,k-1}}}])=[\mathcal{L}_{j,k}]$
		for $2\leq k\leq n_j$,
		$f([\mathcal{L}_{\mathscr{I},\omega_{u_{j,1}}}])=[\mathcal{L}_{j,1}]$
		and
		$f([\mathcal{L}_{\mathscr{I},-\omega_{u_{j,n_j}}}])=[\mathcal{L}_{j,n_{j+1}}]$. Now,
		the presentation \eqref{iso} for the Grothendieck ring of
		$F_\mathscr{I}$ follows by \eqref{e1}, \eqref{e2} and \eqref{e3}.
	\end{proof}

	Let
	$\mathcal{R}'':=\bz[y^{\pm1}_{j1}\,, y^{\pm1}_{j2}|\, 1\leq j\leq r]$
	and $\mathcal{I}''_1$ is the ideal in $\mathcal{R}''$ generated by the
	elements $y_{11}+y_{12}-2$ and $y_{11}\cdot y_{12}-1$. For each
	$2\leq j\leq r$ let $\mathcal{I}''_j$ be the ideal generated by the polynomials
	$\displaystyle y_{j1}+y_{j2}-\prod_{l=1}^{j-1}y_{l1}^{c_{jl}}-1$ and
	$\displaystyle y_{j1}\cdot y_{j2}-\prod_{l=1}^{j-1}y_{l1}^{c_{jl}}$, such that
	\[c_{jl}:=\langle \alpha_{u_{j,1}}\,,\,\alpha^\vee_{u_{l,1}}\rangle
	\text{ for each $l<j$.}\] Let $\mathcal{I}''$ denote the ideal
	$\mathcal{I}''_1+\cdots+\mathcal{I}''_{r}$ in $\mathcal{R}''$.

	\begin{coro}\label{kbs}
		Suppose that the flag Bott-Samelson variety $F_\mathscr{I}$ is a
		Bott-Samelson variety (cf. Corollary \ref{bs}). Then the map from
		$\mathcal{R}''$ to $K^0(F_\mathscr{I})$ which sends $y_{j1}$ to
		$[p_r^*\circ\cdots\circ p_{j+1}^*(\mathcal{L}_{\omega_j})]$ for
		$1\leq j\leq r-1$ and $y_{r1}$ to $[\mathcal{L}_{\omega_{r}}]$ induces
		the following isomorphism of $\mathbb{Z}$-algebras:
		\be\label{bspres} K^0(F_\mathscr{I})\simeq
		\mathscr{K}^0(F_\mathscr{I})\simeq \mathcal{R}''/\mathcal{I}''.\ee
		Using the relations in $\mathcal{I}''$ we can further simplify this
		to \be\label{bssimplepres}\mathscr{K}^*(F_\mathscr{I})\simeq
		\bz[y_{j1}^{\pm1}|\, 1\leq j\leq
		r]/\langle(y_{j1}-1)(y_{j1}-y_{11}^{c_{j1}}\cdots
		y_{j-1,1}^{c_{j,j-1}}); 1\leq j\leq r\rangle,\ee which matches with
		the presentation given by Sankaran and Uma (\cite[Theorem
		5.4]{su2}).
	\end{coro}
	\begin{proof}
		The isomorphism (\ref{bspres}) directly follows from Theorem
		\ref{mainth} and Corollary \ref{bs}.
		
		Now, from the relations in $\mathcal{I}_1''$ we get
		$y_{12}=y_{11}^{-1}$ and from the relations in $\mathcal{I}_j''$ we
		get
		$\displaystyle
		y_{j2}=y_{j1}^{-1}(\prod_{l=1}^{j-1}y_{l1}^{c_{jl}})$. Moreover,
		$y_{j1}+y_{j2}=\prod_{l=1}^{j-1}y_{l1}^{c_{jl}}+1$ for
		$1\leq j\leq r$ (for $j=1$ the right hand side is $2$ so this holds
		by $\mathcal{I}_1''$). This implies that
		\[y_{j1}(\prod_{l=1}^{j-1}y_{l1}^{c_{jl}}+1-y_{j1})-(\prod_{l=1}^{j-1}y_{l1}^{c_{jl}})=0.\]
		This simplifies to
		\[ (y_{j1}-1)(\prod_{l=1}^{j-1}y_{l1}^{c_{jl}})-y_{j1}(y_{j1}-1)=0\]
		which further simplifies to
		\[(y_{j1}-1) ( \prod_{l=1}^{j-1}y_{l1}^{c_{jl}}-y_{j1})=0.\] Thus
		we can reduce the presentation to the simpler form
		\eqref{bssimplepres}.  
	\end{proof}

	{\it Acknowledgement:} The authors are very grateful to
	Prof. P. Sankaran for reading the earlier version of the manuscript
	and for valuable comments and suggestions. The authors are very
	grateful to the referee for a careful reading of the manuscript and
	for valuable comments and suggestions. The first named author thanks
	Indian Institute of Technology, Madras for PhD fellowship. The second
	author was supported by MATRICS SERB project number MTR/2022/000484.

\end{document}